\documentclass[reqno,10pt,a4paper,dvips]{amsart}

\usepackage{amssymb,mathptmx,psfrag,eucal,array,setspace,color,geometry,enumitem}
\usepackage[dvips]{graphicx}
\usepackage{tikz}

\DeclareSymbolFont{largesymbols}{OMX}{zplm}{m}{n} 

\numberwithin{equation}{section}

\newcolumntype{C}{>{$}c<{$}} 

\allowdisplaybreaks

\theoremstyle{plain}
\newtheorem{thm}[]{Theorem}
\newtheorem{defn}{Definition}
\newtheorem{prop}[thm]{Proposition}

\newtheorem{lemma}[thm]{Lemma}
\newtheorem{conj}{Conjecture}

\newtheorem*{result}{Theorem}

\newcommand{\del}{\partial}
\newcommand{\cW}{\mathcal{W}}
\newcommand{\cM}{\mathcal{M}}

\newcommand{\alg}[1]{\mathfrak{#1}}
\newcommand{\group}[1]{\mathsf{#1}}

\newcommand{\func}[2]{#1 \left( #2 \right)}

\newcommand{\brac}[1]{\left( #1 \right)}

\newcommand{\set}[1]{\left\{ #1 \right\}}

\newcommand{\ZZ}{\mathbb{Z}}

\newcommand{\RR}{\mathbb{R}}
\newcommand{\CC}{\mathbb{C}}

\newcommand{\affine}[1]{\widehat{#1}}

\newcommand{\comm}[2]{\bigl[ #1 , #2 \bigr]}

\newcommand{\ket}[1]{\bigl\lvert #1 \bigr\rangle}

\newcommand{\IrrMod}[1]{\mathcal{L}_{#1}}
\newcommand{\TypMod}[1]{\mathcal{E}_{#1}}

\newcommand{\SLA}[2]{\alg{#1} \left( #2 \right)}
\newcommand{\AKMA}[2]{\affine{\alg{#1}} \left( #2 \right)}
\newcommand{\AKMSA}[3]{\affine{\alg{#1}} \left( #2 \middle\vert #3 \right)}
\newcommand{\SLSA}[3]{\alg{#1} \left( #2 \middle\vert #3 \right)}
\newcommand{\SLSG}[3]{\group{#1} \left( #2 \middle\vert #3 \right)}

\newcommand{\ch}[1]{\mathrm{ch} \bigl[ #1 \bigr]}

\newcommand{\fuse}{\mathbin{\times}}

\newcommand{\normord}[1]{\mbox{${} : #1 : {}$}} 

\newcommand{\FS}[1]{W^{\left( 2 \right)}_{#1}}

\newcommand{\secref}[1]{Section~\ref{#1}}

\DeclareMathOperator{\tr}{tr}

\DeclareMathOperator{\Com}{Com}

\newcommand{\dz}{\mathrm{d}}

\begin{document}

\title[Coset Constructions of Logarithmic $(1,p)$-Models]{Coset Constructions of Logarithmic $(1,p)$-Models}

\author[T Creutzig]{Thomas Creutzig}

\address[Thomas Creutzig]{
Fachbereich Mathematik \\
Technische Universit\"{a}t Darmstadt\\
Schlo\ss{}gartenstra\ss{}e 7\\
64289 Darmstadt\\ Germany
}

\email{tcreutzig@mathematik.tu-darmstadt.de}

\author[D Ridout]{David Ridout}

\address[David Ridout]{
Department of Theoretical Physics \\
Research School of Physics and Engineering;
and
Mathematical Sciences Institute;
Australian National University \\
Canberra, ACT 2600 \\
Australia
}

\email{david.ridout@anu.edu.au}

\author[S Wood]{Simon Wood}

\address[Simon Wood]{
Kavli Institute for the Physics and Mathematics of the Universe (WPI)\\
Todai Institutes for Advanced Study\\
The University of Tokyo\\
1-5, Kashiwanoha 5-Chome
Kashiwa-shi, Chiba 277-8583\\
Japan
}

\email{simon.wood@ipmu.jp}

\thanks{\today}

\begin{abstract}
One of the best understood families of logarithmic conformal field theories is that consisting of the $(1,p)$ models ($p = 2, 3, \ldots$) of central charge $c_{1,p} = 1 - 6 \brac{p-1}^2 / p$.  This family includes the theories corresponding to the singlet algebras $\cM(p)$ and the triplet algebras $\cW(p)$, as well as the ubiquitous symplectic fermions theory.  In this work, these algebras are realized through a coset construction.

The $W^{(2)}_n$-algebra of level $k$ was introduced by Feigin and Semikhatov as a (conjectured) quantum hamiltonian reduction of $\AKMA{sl}{n}_k$, generalising the Bershadsky-Polyakov algebra $W^{(2)}_3$.  Inspired by work of Adamovi\'{c} for $p=3$, vertex algebras $\mathcal B_p$ are constructed as subalgebras of the kernel of certain screening charges acting on a rank $2$ lattice vertex algebra of indefinite signature.  It is shown that for $p \leqslant 5$, the algebra $\mathcal B_p$ is a quotient of $W^{(2)}_{p-1}$ at level $-(p-1)^2/p$ and that the known part of the operator product algebra of the latter algebra is consistent with this holding for $p>5$ as well.  The triplet algebra $\cW(p)$ is then realised as a coset inside the full kernel of the screening operator, while the singlet algebra $\cM(p)$ is similarly realised inside $\mathcal B_p$.  As an application, and to illustrate these results, the coset character decompositions are explicitly worked out for $p=2$ and $3$.

\vspace*{2mm}
\noindent {MSC}: 17B68, 17B69

\noindent {Keywords}: Logarithmic conformal field theory, vertex algebras
\end{abstract}

\maketitle

\onehalfspacing

\section{Introduction} \label{sec:Intro}

The principal examples of logarithmic conformal field theories are the families associated to affine superalgebras, to admissible level affine algebras, and to the kernels of screenings acting on lattice theories.  In all three families, only a few examples are well-understood in the sense that the representation theory has been worked out in detail.  These examples include those associated to the $A_1$-root lattice, the logarithmic $(q,p)$ minimal models \cite{Kau,GK,FHST,FGST,GR,AM1,AM2,TW1,TW2}, the $\SLSG{GL}{1}{1}$ WZNW theory \cite{SS,CS,CRo,CR2,CR3}, and the admissible level theories of $\AKMA{sl}{2}$ at $k=-1/2$ \cite{R1,R2,R3,CR1} and $k=-4/3$ \cite{G,A2,AM1,CR1}.  The representation theory of the $(1,p)$ series is very similar to that of $\AKMSA{gl}{1}{1}$ and $\AKMA{sl}{2}$ at admissible levels.  Indeed, there are several relationships known between the ``smallest'' members of each logarithmic family:  The logarithmic $(1,2)$-model may be described as a coset of a simple current extension\footnote{
For the purposes of this article, a simple current may be defined to be a simple module which has an inverse in the fusion ring.} of $\AKMA{sl}{2}$ at level $-1/2$ \cite{R2}, while the $(1,3)$-model is a coset of $\AKMA{sl}{2}_{-4/3}$ \cite{A2}.  Moreover, $\AKMA{sl}{2}_{-1/2}$ is itself realisable as a coset of (an extension of) $\AKMSA{gl}{1}{1}$ \cite{CR3}.
  
The purpose of this work is to extend this picture by providing coset constructions for the $(1,p)$ singlet algebras $\cM(p)$ and triplet algebras $\cW(p)$, for all $p$. For this, the crucial hint is the work \cite{FS} of Feigin and Semikhatov on algebras denoted by $W^{(2)}_n$, which generalise the well-known Bershadsky-Polyakov algebra \cite{B,P}. These algebras are constructed in two ways, first as a kernel of screenings associated with the quantum group of $\SLSA{sl}{n}{1}$ and second as a subalgebra of $\AKMSA{sl}{n}{1}_k \otimes V_L$ commuting with the subalgebra $\AKMA{sl}{n}_k \otimes \AKMA{gl}{1}$.  Here, $V_L$ is a rank one lattice vertex algebra and the affine vertex superalgebras are the universal ones of the indicated levels $k$. We have the following picture in mind:
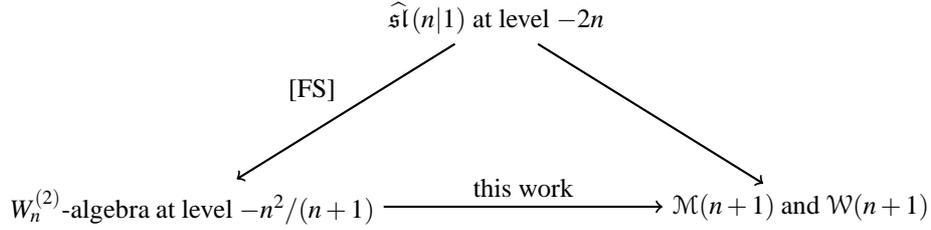
\begin{equation*}
\parbox[c]{0.83\textwidth}{
\begin{center}
\begin{tikzpicture}[auto,thick,
	nom/.style={rectangle,draw=black!20,fill=black!20,minimum height=20pt}
	]
\node (botleft) at (0,-1.25) [] {$W^{(2)}_n$-algebra at level $-n^2/(n+1)$};
\node (topmiddle) at (4,1.25) [] {$\AKMSA{sl}{n}{1}$ at level $-2n$};
\node (botright) at (8,-1.25) [] {$\cM(n+1)$ and $\cW(n+1)$};
\draw [->] (topmiddle) to node [above left] {\cite{FS}} (botleft);
\draw [->] (topmiddle) to node [above] {} (botright);
\draw [->] (botleft) to node [above] {this work} (botright);
\end{tikzpicture}
\end{center}
} .
\end{equation*}
One arrow of this diagram was essentially explained by Feigin and Semikhatov. In the present work, we are interested in a coset construction starting from the $W^{(2)}_n$-algebras and yielding the vertex algebras $\cM(p)$ and $\cW(p)$ of the logarithmic $(1,p)$-models.  This generalises the results for $p=2$ and $3$ mentioned above.

In order to understand the relation between the $W^{(2)}_n$-algebras and the $(1,n+1)$-theories, one needs a suitably explicit description of the $W^{(2)}_n$-algebras at level $-n^2/(n+1)$. Adamovi\'{c} \cite{A2} provides such a description for $W^{(2)}_2 \equiv \AKMA{sl}{2}$ (the level is then $-4/3$). Recall that the $(1,p)$-triplet algebra is constructed as the kernel of a screening inside an appropriate rank one lattice algebra associated to the (rescaled) $A_1$ root lattice. Adamovi\'{c} considers a rank two lattice of indefinite signature, whose associated lattice vertex algebra contains the rank one lattice vertex algebra of the $(1,3)$-theory as a subalgebra. For the screening charge, he chooses that of the $(1,3)$-theory so as to guarantee that the kernel contains the $(1,3)$-triplet algebra $\cW(3)$ as a subalgebra. But, he also finds the simple affine vertex algebra of $\SLA{sl}{2}$ at level $-4/3$ as a subalgebra of the screening's kernel.

Our first result generalises this. We consider an appropriate rank two lattice $D$ of indefinite signature, such that the lattice of the $(1,p)$-triplet theory is a sublattice. We choose the screening charge to be that of the $(1,p)$-triplet theory so that the $(1,p)$-triplet algebra is contained in the screening's kernel. In addition, we find another subalgebra, which we call $\mathcal B_p$, that is generated by two fields of conformal dimension $n/2$. We compute the operator product algebra of $\mathcal B_p$ and also some relations in $\mathcal B_5$. The result can be summarized as 

\begin{result}
For $p=2,3,4,5$, the algebra $\mathcal B_p$ is a quotient of $W^{(2)}_{p-1}$ at level $-(p-1)^2/p$. 
In general, comparing operator product algebras is consistent with the conjecture that $\mathcal B_p$ is a quotient of $W^{(2)}_{p-1}$ at level $-(p-1)^2/p$ for all $p$, 
\end{result}

\noindent As the operator product algebra of the $W^{(2)}_n$-algebra for $n\geq 4$ is only partially known (see \cite{FS}), we are unable to make stronger statements concerning the relationship between these algebras and the $\mathcal B_p$. We remark, however, that the dimension three Virasoro primary field of $W^{(2)}_n$ (that appears for $n>3$) is in the kernel of the proposed homomorphism.

Now that we have an explicit description of the algebra $\mathcal B_p$, we investigate its coset algebras. In general by this we mean the following:

\begin{defn}
Let $A$ be a vertex algebra and $B \subseteq A$ a subalgebra. Then, the \emph{coset algebra} of $B$ in $A$ is the commutant subalgebra $\Com(B,A) \subseteq A$. In physics, the conformal field theory corresponding to the coset algebra is usually denoted by
\[
\frac{A}{B}.
\]
If $B = \Com(\Com(B,A),A)$, then $B$ and $\Com(B,A)$ are said to form a \emph{Howe pair} inside $A$. 
\end{defn} 

\noindent Mutually commuting pairs in the theory of vertex algebras have been introduced in \cite{DM,LL}, and examples
containing the singlet algebras $\cM(2)$ and $\cM(3)$ appear in \cite{L,CL} and \cite{A2}, respectively.
Our main result is then 

\begin{result}
Within the kernel of the screening operator, the $(1,p)$-triplet algebra $\cW(p)$ and a certain rank one lattice vertex algebra form a Howe pair.  Furthermore, the $(1,p)$-singlet algebra $\cM(p)$ and a certain rank one Heisenberg vertex algebra form a Howe pair inside $\mathcal B_p$.
\end{result}

\noindent It is somewhat remarkable that the very explicit descriptions of the algebras involved allow us to exhaustively describe these commutants.

Given a vertex operator algebra together with a mutually commuting pair of subalgebras, an important question is how a given vertex algebra module will decompose into modules of the two subalgebras. Consider the cases $p=2$, for which $\mathcal B_2$ is the rank one $\beta\gamma$ vertex algebra, and $p=3$, for which $\mathcal B_3$ is the simple affine vertex algebra of $\AKMA{sl}{2}_{-4/3}$. In both cases, characters are known for the full spectrum of modules and extended algebras corresponding to simple currents are known \cite{CR1}. As an application of our results, we decompose characters of all the irreducible $\mathcal B_p$-modules, for $p=2,3$, into irreducible characters of $\cM(p)$ and the appropriate rank one Heisenberg algebra. In addition, we find the simple current extensions whose modules' characters can be decomposed into those of $\cW(p)$ and the appropriate rank one lattice vertex operator algebra. When $p=2$, $\mathcal B_2$ is itself a simple current extension
of $\AKMA{sl}{2}_{-1/2}$ \cite{R3} and we provide character decompositions for the latter into $\cM(2)$- and $\cW(2)$-characters.

The article is organized as follows. In section two, we provide necessary information concerning the triplet algebra $\cW(p)$ and the singlet algebra $\cM(p)$. The main results are then proven in section three, where we first construct the vertex algebra $\mathcal B_p$, compute the first few leading terms of its operator product algebra, and compare the result with that of the $W^{(2)}_{p-1}$-algebra at level $-(p-1)^2/p$. The second part of this section then proves that $\cW(p)$ may be realised as a coset algebra inside the kernel of a screening operator, while $\cM(p)$ may be realised as a coset algebra inside $\mathcal B_p$. Section four then details the character decompositions that illustrate our results for $p=2$ and $p=3$.

\subsection*{Acknowledgements} We thank Andrew Linshaw and Antun Milas for
carefully reading the manuscript and useful discussions related to this work.  We also thank the referee for their very helpful comments.  DR's research is supported by an
Australian Research Council Discovery Project DP1093910.
SW's work is supported by the World Premier International Research Center
Initiative (WPI Initiative), MEXT, Japan; the Grant-in-Aid for JSPS Fellows
number 2301793; and the JSPS fellowship for foreign researchers P11793.

\section{$\mathcal{W}(p)$ and $\cM(p)$ theories} \label{sec:WnM}

In this section, we outline the representation theory of the \(\mathcal{W}(p)\) and \(\mathcal{M}(p)\) vertex operator algebras --- to be defined in the following.  This summary is based on \cite{A1} and \cite{NT}.

\subsection{The free boson or Heisenberg vertex operator algebra}

The Heisenberg vertex operator algebra is that 
whose field modes are given by sums of products of
generators of the Heisenberg algebra \(\mathcal{H}\). This algebra is an associative complex algebra
generated by an infinite number of generators \(a_n, n\in\mathbb{Z}\), satisfying the commutation relations
\[
  [a_m,a_n]=m\delta_{m+n,0} 1 \,.
\]
The Heisenberg algebra contains a number of commutative subalgebras. The most important one for this paper is
\[
\mathcal{H}^{\geq}=\mathbb{C}[a_0,a_1,a_2,\dots]\,.
\]
The highest weight representations \(\mathcal{F}_\lambda\) of \(\mathcal{H}\)
are called Feigin-Fuchs modules or Fock spaces.
They are
uniquely characterised by their Heisenberg highest
weight \(\lambda\in\mathbb{C}\).\footnote{For physics applications, one usually restricts oneself to real \(\lambda\), which is what we will do in later sections.} If we denote the highest weight state by
\(|\lambda\rangle\in\mathcal{F}_\lambda\), so that
\begin{align}
  a_n|\lambda\rangle=\delta_{n,0} \lambda|\lambda\rangle\,,\quad n\geq 0\,,
\end{align}
then \(\mathcal{F}_\lambda\) can be constructed as
\begin{align}
  \mathcal{F}_\lambda=\mathcal{H}\otimes_{\mathcal{H}^{\geq}}\CC|\lambda\rangle\,.
\end{align}

The weight 0 Fock space \(\mathcal{F}_0\) carries the structure of a vertex operator algebra -- the so called
Heisenberg vertex operator algebra.
As a vertex operator algebra, \(\mathcal{F}_0\) is generated by the field
\begin{align}
  a(z)=\sum_{n\in \mathbb{Z}}a_n z^{-n-1}\,,
\end{align}
which satisfies the operator product expansion
\begin{align}
  a(z)a(w)\sim \frac{1}{(z-w)^2}\,.
\end{align}
The choice of conformal structure is not unique. For any \(\alpha_0\in\mathbb{C}\), one can define a Virasoro field
\begin{align}
  T(z)=\frac{1}{2}:a(z)^2:+\frac{\alpha_0}{2}\partial a(z)\,,
\end{align}
where \(:\cdots :\) denotes normal ordering, meaning that one arranges 
the Heisenberg generators by ascending mode number.
The central charge defined by this choice of Virasoro field is
\begin{align}
  c_{\alpha_0}=1-3\alpha_0^2\,.
\end{align}

The primary fields corresponding to the highest weight states \(|\lambda\rangle\in\mathcal{F}_\lambda\) are constructed 
by means of an 
auxiliary field which is the formal primitive of \(a(z)\):
\begin{align}
  \phi(z)=\hat a +a_0\log z -\sum_{n\neq0}\frac{a_n}{n}z^{-n}\,.
\end{align}
Here,
\begin{align}
  [a_m,\hat a]=\delta_{m,0} 1 \,.
\end{align}
Exponentials of the auxiliary generator \(\hat a\) shift the weight of the Fock spaces, defining maps 
\begin{align}
  e^{\mu\hat a} \colon \mathcal{F}_{\lambda}\longrightarrow \mathcal{F}_{\lambda+\mu}\,.
\end{align}
The primary field corresponding to the state \(|\lambda\rangle\) is given by
\begin{align}\label{eq:screeningop}
  V_\lambda(z)={}:e^{\lambda \phi(z)}:{}=e^{\lambda\hat a}z^{\lambda a_0} e^{\lambda\sum_{n\geq 1}a_{-n}z^n/n}e^{-\lambda\sum_{n\geq1}a_{n}z^{-n}/n}\,.
\end{align}
The conformal weight 
of this primary field
is
\begin{align}\label{eq:weightformula}
  h_\lambda=\frac{\lambda}{2}(\lambda-\alpha_0)\,.
\end{align}

\subsection{The lattice vertex operator algebra \(\mathcal{V}(p)\)}

For special values of \(\alpha_0\), one can define a 
lattice vertex operator algebra \(\mathcal{V}(p)\). Let
\(p\) be an integer greater than one and define \(\alpha_+=\sqrt{2p}\), \(\alpha_-=-\sqrt{2/p}\) and
\begin{align}\label{alphaformulae}
  \alpha_{r,s}=\frac{1-r}{2}\alpha_++\frac{1-s}{2}\alpha_-\,,
\end{align}
where \(r\) and \(s\) are integers.
Note that \(\alpha_{r,s}\) is periodic:  \(\alpha_{r,s}=\alpha_{r+1,s+p}\). We set the parameter \(\alpha_0\) of the Heisenberg
vertex operator algebra to \(\alpha_0=\alpha_++\alpha_-\), so that the Virasoro field is given by
\begin{align}
  T(z)=\frac{1}{2}:a(z)^2:+\frac{p-1}{\sqrt{2p}}\partial a(z)
\end{align}
and the central charge by
\begin{align}
  c_p=1-6\frac{(p-1)^2}{p}\,.
\end{align}

We introduce the lattices
\begin{align}
  L=\mathbb{Z}\alpha_+, \qquad
  L^\vee=\hom_{\mathbb{Z}}(L,\mathbb{Z})=\mathbb{Z}\frac{\alpha_-}{2}\,.
\end{align}
Then, \(\alpha_{r,s}\in L^\vee\) for all \(r,s\in\mathbb{Z}\). The lattice algebra \(\mathcal{V}(p)\) is an extension
of the Heisenberg vertex operator algebra \(\mathcal{F}_0\) which, as a
Heisenberg module,
 is given by an infinite sum 
of Fock spaces: 
\begin{align}\label{eq:Vpsum}
  \mathcal{V}(p)=\bigoplus_{\lambda\in L} \mathcal{F}_\lambda\,.
\end{align}
There are \(2p\) isomorphism classes of irreducible 
\(\mathcal{V}(p)\)-modules \cite{D}. They are parametrised by the cosets \([\mu]\in L^\vee/L\):
\begin{align}\label{eq:Vmodsum}
  \mathcal{V}_{[\mu]}=\bigoplus_{\lambda\in [\mu]}\mathcal{F}_{\lambda}\,.
\end{align}
If we label the simple \(\mathcal{V}(p)\)-modules by \(\alpha_{r,s}\), for \(r=1,2,\ 1\leq s\leq p\), then the definitions
\eqref{eq:Vpsum} and \eqref{eq:Vmodsum} can be reexpressed as
\begin{align}
  \mathcal{V}_{[\alpha_{r,s}]}=\bigoplus_{n\in\mathbb{Z}}\mathcal{F}_{\alpha_{r+2n,s}}\,.
\end{align}

In more physical terms, $\mathcal{V}(p)$ is the extension of $\mathcal{F}_0$, or rather its associated vertex operator algebra, by the simple current group generated by $\mathcal{F}_{\alpha_+}$ under fusion.  It is easy to check that the extension fields are all mutually bosonic and that their conformal dimensions are integers.  The reduction from a continuous spectrum to a finite spectrum may be explained by noting that the constraint on $[\mu] \in L^{\vee} / L = \set{0, -\tfrac{1}{2} \alpha_-, -\alpha_-, \ldots, -\tfrac{1}{2} \brac{2p-1} \alpha_-}$ in the extended algebra module $\mathcal{V}_{[\mu]}$ arises from requiring that the conformal dimensions of the fields of $\mathcal{V}_{[\mu]}$ all differ from one another by integers.  These modules therefore constitute the untwisted sector of the extended theory.

\subsection{Screening operators and the singlet and triplet algebras} \label{sec:Screening}

By the formula \eqref{eq:weightformula} for conformal weights, there are two primary weight 1 fields
\begin{align}
  Q_\pm(z)=V_{\alpha_{\pm}}(z)\,,
\end{align}
which can be used to construct 
screening operators, though we will only be using \(Q_-(z)\) for this purpose here.
The singlet vertex operator algebra \(\mathcal{M}(p)\) is defined to be the
vertex operator subalgebra of \(\mathcal{F}_0\) given by
\begin{align}
  \ker \left(\oint Q_-(z)\,\dz z \colon 
  \mathcal{F}_0\longrightarrow \mathcal{F}_{\alpha_-}\right)\,,
\end{align}
while the triplet vertex operator algebra \(\mathcal{W}(p)\) is the vertex operator subalgebra of \(\mathcal{V}(p)\) given by
\begin{align}
  \ker \left(\oint Q_-(z)\,\dz z \colon  
  \mathcal{V}_{[0]}\longrightarrow \mathcal{V}_{[\alpha_-]}\right)\,.
\end{align}

As a vertex operator algebra, \(\mathcal{W}(p)\) is generated by the Virasoro
field \(T(z)\) it inherits from 
$\mathcal{V}_{[0]}$ and two additional weight
\(2p-1\) Virasoro primary fields \(W^{\pm}(z)\).
These two weight \(2p-1\) fields generate an additional weight \(2p-1\)
Virasoro primary field \(W^0(z)\) in their operator product expansion, hence the name 
``triplet algebra''.
As a Virasoro module, \(\mathcal{W}(p)\) decomposes into an infinite direct sum of irreducible Virasoro modules:
\begin{align}
  \mathcal{W}(p)=\bigoplus_{n\geq 0}(2n+1)L(h_{\alpha_{2n+1,1}},c_p)\,.
\end{align}
Here,  
\(L(h,c)\) is the irreducible Virasoro module of weight \(h\) and central charge \(c\).

The singlet algebra \(\mathcal{M}(p)\) is not only a vertex operator subalgebra of \(\mathcal{F}_0\), but also of \(\mathcal{W}(p)\).
In fact it can alternatively 
be defined as
\begin{align}
  \mathcal{M}(p)=\mathcal{F}_0\cap \mathcal{W}(p)\,.
\end{align}
As a vertex operator algebra, \(\mathcal{M}(p)\) is generated by the Virasoro field \(T(z)\) and the weight \(2p-1\) Virasoro primary field
\(W^0(z)\), hence the name ``singlet algebra''. As a Virasoro module, \(\mathcal{M}(p)\) decomposes into an infinite direct 
sum of the same irreducible Virasoro modules as \(\mathcal{W}(p)\):
\begin{align}
  \mathcal{M}(p)=\bigoplus_{n\geq 0}L(h_{\alpha_{2n+1,1}},c_p)\,.
\end{align}

In order to understand the representation theories of \(\mathcal{M}(p)\) and \(\mathcal{W}(p)\), we need to refine our understanding of
the screening operators somewhat. The main difficulty arises from the fact that the factor \(z^{\alpha_-a_0}\) in \(Q_-(z)\) 
(see formula \eqref{eq:screeningop}) will give rise to non-trivial monodromies when applied to general \(\mathcal{F}_{\alpha_{r,s}}\).
This problem can be circumvented by considering products of the \(Q_-(z)\):
\begin{align}
  \int_{[\Gamma_s]}Q_-(z_1)\cdots Q_-(z_s) \: \dz z_1\cdots\dz z_s \colon  
  \mathcal{F}_{\alpha_{r,s}}\rightarrow \mathcal{F}_{\alpha_{r,-s}}\,.
\end{align}
The cycle \([\Gamma_s]\) over which this integral is taken is uniquely determined (up to normalisation) by requiring
that the above map be non-trivial (see \cite{NT} for details). This map will henceforth be denoted by 
\(Q_-^{[s]}\). The maps \(Q_-^{[s]}\), \(1\leq s\leq p-1\),
commute with \(\mathcal{M}(p)\) and \(\mathcal{W}(p)\) and therefore define \(\mathcal{M}(p)\)- and \(\mathcal{W}(p)\)-module homomorphisms.

The  
modules $\mathcal{F}_{\alpha_{r,s}}$ with $r \in \ZZ$ and $1 \leq s \leq p-1$ may be organised into Felder complexes under the action of \(Q_-^{[s]}\):
\begin{equation} \label{eq:MFeldercomplex}
\cdots \longrightarrow \mathcal{F}_{\alpha_{r,s}} \xrightarrow{Q_-^{[s]}} \mathcal{F}_{\alpha_{r+1,p-s}} \xrightarrow{Q_-^{[p-s]}} \mathcal{F}_{\alpha_{r+2,s}} \xrightarrow{Q_-^{[s]}} \mathcal{F}_{\alpha_{r+3,p-s}} \longrightarrow \cdots\,.
\end{equation}
Here we have made use of the periodicity \(\alpha_{r,s}=\alpha_{r+1,s+p}\). These sequences are exact, meaning that 
\begin{align}\label{eq:relkerim}
  \text{im} (Q_-^{[p-s]} \colon 
  \mathcal{F}_{\alpha_{r-1,p-s}}\rightarrow\mathcal{F}_{\alpha_{r,s}})=
  \ker (Q_-^{[s]} \colon 
  \mathcal{F}_{\alpha_{r,s}}\rightarrow\mathcal{F}_{\alpha_{r+1,p-s}})\,,
\end{align}
and also extend to \(\mathcal{V}(p)\) modules:
\begin{equation} \label{eq:WFeldercomplex}
\cdots \longrightarrow \mathcal{V}_{[\alpha_{1,s}]} \xrightarrow{Q_-^{[s]}} \mathcal{V}_{[\alpha_{2,p-s}]}\xrightarrow{Q_-^{[p-s]}} \mathcal{V}_{[\alpha_{1,s}]} \xrightarrow{Q_-^{[s]}} \mathcal{V}_{[\alpha_{2,p-s}]} \longrightarrow \cdots\,.
\end{equation}

There are \(2p\) isomorphism classes of simple \(\mathcal{W}(p)\)-modules \(W_{r,s}\), \(r=1,2\) and \(1\leq s\leq p\).
They can be simply characterised 
in terms of the exact sequences \eqref{eq:WFeldercomplex}:
\begin{equation}
W_{r,s} = 
\begin{cases}
\ker \brac{Q_-^{[s]} \colon \mathcal{V}_{[\alpha_{r,s}]} \longrightarrow \mathcal{V}_{[\alpha_{3-r,p-s}]}} & \text{if $1 \leq s < p$,} \\
\mathcal{V}_{[\alpha_{r,p}]} & \text{if $s = p$.}
\end{cases}
\end{equation}
We therefore obtain short exact sequences for $1\leq s < p$:
\begin{equation}
0 \longrightarrow W_{r,s} \longrightarrow \mathcal{V}_{[\alpha_{r,s}]} \longrightarrow W_{3-r,p-s}  \longrightarrow 0.
\end{equation}
The highest conformal weight of \(W_{r,s}\) is \(h_{\alpha_{r,s}}\). For
\(r=1\), the
``space of ground states'' --- the space annihilated by all positive modes of
\(\mathcal{W}(p)\) --- is one-dimensional and
for \(r=2\), it is two-dimensional. 
As Virasoro modules, the \(W_{r,s}\) decompose into an infinite direct sum of simple Virasoro modules:
\begin{align}
    W_{r,s}&=\bigoplus_{n\geq 0}(2n+r)L(h_{\alpha_{2n+r,s}},c_p) 
    \qquad \text{($r=1,2$, $1 \leq s \leq p$).}
\end{align}
The characters of the simple Virasoro modules that constitute the \(W_{r,s}\) are
well-known, leading to explicit expressions for the characters of the latter modules:
\begin{align}
  \ch{W_{r,s}}&=\frac{1}{\eta(q)}\sum_{n\geq 0}
  (2n+r) \brac{q^{((2n+r)p-s)^2 / 4p}-q^{((2n+r)p+s)^2 / 4p}}\,.
\end{align}

The representation theory of the singlet algebra $\mathcal{M}(p)$ is slightly more complicated because there are uncountably many isomorphism classes
of simple modules. For \(\lambda\in\mathbb{C}\setminus L^\vee\), the 
Fock space \(\mathcal{F}_\lambda\) is simple as a 
Virasoro module and therefore also as an \(\mathcal{M}(p)\)-module. 
For \(\lambda\in L^\vee\), the $\mathcal{F}_{\lambda}$ 
are not always semisimple as \(\mathcal{M}(p)\)-modules, but they may again be used to
characterise the simple (highest weight) $\mathcal{M}(p)$-modules:
\begin{equation}\label{eq:MinF}
M_{r,s} = 
\begin{cases}
\ker \brac{Q_-^{[s]} \colon \mathcal{F}_{\alpha_{r,s}} \longrightarrow \mathcal{F}_{\alpha_{r+1,p-s}}} & \text{if $r \geq 1$ and $1 \leq s < p$,} \\
\text{im} \brac{Q_-^{[p-s]} \colon \mathcal{F}_{\alpha_{r-1,p-s}} \longrightarrow \mathcal{F}_{\alpha_{r,s}}} & \text{if $r \leq 1$ and $1 \leq s < p$,} \\
\mathcal{F}_{\alpha_{r,p}} & \text{if $s = p$.}
\end{cases}
\end{equation}
Note that the equality \eqref{eq:relkerim} accounts for the case $r=1$ in this characterisation. 
This time, the short exact sequences (for $1\leq s < p$) take the form
\begin{equation}
0 \longrightarrow M_{r,s} \longrightarrow \mathcal{F}_{[\alpha_{r,s}]} \longrightarrow M_{r+1,p-s}  \longrightarrow 0.
\end{equation}
Again, the highest weight of $M_{r,s}$ is \(h_{\alpha_{r,s}}\). 
As Virasoro modules, these \(\mathcal{M}(p)\)-modules decompose as
\begin{equation}
\begin{split}
M_{r,s}&=\bigoplus_{k\geq0}\mathcal{L}(h_{\alpha_{r+2k,s}}) \qquad \text{($r \geq 1$, $1 \leq s \leq p$),} \\
M_{r+1,p-s}&=\bigoplus_{k\geq0}\mathcal{L}(h_{\alpha_{r-2k,s}}) \qquad \text{($r \leq 0$, $1 \leq s \leq p$).}
\end{split}
\end{equation}
We remark 
that for \(r\geq 1\) and \(1\leq s\leq p\), the \(\mathcal{M}(p)\)-modules \(M_{r,s}\) 
and \(M_{2-r,s}\) are
isomorphic as Virasoro modules but not as \(\mathcal{M}(p)\)-modules. 
The simple \(\mathcal{W}(p)\)-modules are semisimple as \(\mathcal{M}(p)\)-modules:
\begin{equation}\label{eq:tripletdecomp}
\begin{split}
  W_{r,s}&=\bigoplus_{k\in\ZZ} \ M_{2k+r,s}\,.
\end{split}
\end{equation}
For \(r\geq 1\) and \(1\leq s\leq p-1\), the characters of the singlet modules are given by
\begin{equation} \label{eq:singletchars}
\begin{split}
  \ch{M_{r,s}}=\ch{M_{2-r,p-s}}&=\frac{1}{\eta(q)}\sum_{n\geq 0}
  (q^{((r+2n)p-s)^2/4p}-q^{((r+2n)p+s)^2/4p})\\
  &= \sum_{n\geq 0}\Bigl(\ch{\mathcal{F}_{\alpha_{r-2n-1,p-s}}}-
  \ch{\mathcal{F}_{\alpha_{r-2n-2,s}}}\Bigr)\,.
\end{split}
\end{equation}

\section{Coset constructions for $\cM(p)$ and $\cW(p)$} \label{sec:Commutants}

In this section, we will construct a family of free field vertex algebras $\mathcal B_p$ inside a rank two lattice algebra. These vertex algebras will be compared with the W-algebras $W_n^{(2)}$ introduced by Feigin and Semikhatov and the singlet algebras $\cM(p)$ and triplet algebras $\cW(p)$ will be characterized as commutant subalgebras. 

\subsection{The Feigin-Semikhatov algebras $W_n^{(2)}$}

In \cite{FS}, Feigin and Semikhatov introduce a family of $W$-algebras associated to the affine Lie superalgebra 
$\AKMSA{sl}{n}{1}$. They provide two constructions of
these algebras. The first is as the intersection of kernels of 
screening charges inside a certain lattice vertex algebra, where the screening charges are associated to a simple root 
system of $\SLSA{sl}{n}{1}$. The second is as a subalgebra of the tensor product of the universal affine vertex super algebra of $\AKMSA{sl}{n}{1}$ at level $k'=1/(k+n-1)+1-n$ with a rank one lattice vertex algebra $V$.
They use these constructions to compute the first few leading terms of the operator product algebra, but for general $n$, a complete characterization of the algebra is unknown.
We will use the second 
to define a universal version of this algebra which we shall refer to as the $\FS{n}$-algebra of level $k$.

Feigin and Semikhatov consider two fields 
$\widetilde{E}(z)$, $\widetilde{F}(z)$ in this tensor product theory and define the $\widetilde{W}_n^{(2)}$-algebra of level $k$ as the algebra 
generated by these two fields under iterated operator products. A set of strong generators\footnote{We recall that strongly generated means that every field of the algebra is a normally-ordered polynomial in the strong generators and their derivatives and that being freely generated means that there are no relations between generators --- there is no non-trivial linear combination of 
normally-ordered products of the generators and their derivatives which vanishes.} can then be constructed step by step
by simply adding those fields in a given operator product expansion that cannot
be expressed in terms of the previous fields. At each step, 
one has to ensure that there are no dependencies, that is, that
the set of generators obtained thus far is actually strong.
It is clear that this procedure endows the resulting vertex algebra with 
a countable ordered set of strong generators. Call this set $S$ and the generators $\{\widetilde{X}^{(i)}(z)\}_{i\in S}$. It contains a Virasoro field $\widetilde{T}(z)$ for which $\widetilde{E}(z)$ and $\widetilde{F}(z)$ are both primary fields of conformal dimension $n/2$.
Our definition of the $\FS{n}$-algebra is then as follows:
\begin{defn}
Let $\{ X^{(i)}(z)\}_{i\in S}$ be a
set of fields whose operator product expansions have singular terms which are 
identical to those 
of the $\widetilde{X}^{(i)}(z)$ (but omitting the tildes).
Let $A$ be the vector space spanned
  by the vacuum vector $|0\rangle$ and the ordered PBW-like basis of non-zero elements of the form
\[  X^{(i_1)}_{j_1}\dots  X^{(i_m)}_{j_m} |0\rangle,   \]
where $j_1\leq j_2\leq \dots\leq j_m < 0$ and if $j_{r}=j_{r+1}$ then $i_r
\leq i_{r+1}$.
Then, $T(z)$ is a Virasoro field and, by the reconstruction theorem, the vector
space $A$ can be given the structure of a vertex operator algebra, 
the Feigin-Semikhatov algebra $\FS{n}$ of level $k$.   
\end{defn}

\noindent By construction, 
the $\widetilde{W}_n^{(2)}$-algebra of level $k$ is a quotient of $\FS{n}$ of level $k$ because there may exist relations captured in the regular terms of the operator product expansions of the strong generators $\widetilde{X}^{(i)}(z)$.  Our definition therefore gives a universal version of the algebras constructed by Feigin and Semikhatov.

We remark that Feigin and Semikhatov originally proposed the notation, suggesting that the ``$(2)$'' of $W_n^{(2)}$ indicates 
that this algebra behaves similarly to $\SLA{sl}{2}$. They implicitly assume, except for $n=1$, that the $\FS{n}$ algebra of level $k$ is a quantum Hamiltonian 
reduction corresponding to a certain non-principal embedding of $\SLA{sl}{2}$ into the universal 
affine vertex algebra $\AKMA{sl}{n}_k$. 
Since this is an unproven statement, let us state it as a conjecture.
Let $\varphi_n:\SLA{sl}{2}\rightarrow \SLA{sl}{n}$ be an embedding of $\SLA{sl}{2}$ in $\SLA{sl}{n}$, 
such that $\SLA{sl}{n}$ decomposes into $\SLA{sl}{2}$ modules as
\begin{equation}
\lambda_{n-1}\oplus \lambda_{n-1}\oplus \bigoplus_{i=1}^{n-1} \lambda_{2i-1},
\end{equation}
where $\lambda_m$ denotes the $m$-dimensional irreducible representation of $\SLA{sl}{2}$.
Let $\mathcal W^{(2)}_n$ of level $k$ be the quantum Hamiltonian reduction of the affine vertex algebra of $\SLA{sl}{n}$
of level $k$ corresponding to the embedding $\varphi_n$.
These types of quantum Hamiltonian reduction can be found for example in \cite{KRW}.
\begin{conj}\label{conj:qh}
For $n>1$, the two algebras $\mathcal W^{(2)}_n$ and
$W_n^{(2)}$, both of level $k$, are isomorphic.
\end{conj}
This conjecture is implicitely stated in \cite{FS}. We remark that our results do not depend on this conjecture. 
Nonetheless, let us state some observations in its favour. 
The quantum Hamiltonian reduction gives rise to a vertex algebra that is generated 
by two bosonic fields $E$ and $F$ of conformal dimension $n/2$.  Moreover, this reduction is strongly and freely generated 
as a vertex algebra by two bosonic fields of dimension $n/2$ and one each of dimensions $1,2,\ldots,n-1$.

There is another algebra whose operator product algebra coincides with the known part of the operator product expansions of the $\FS{n}$ algebra at the critical level $k=-n$.
This algebra is realised as a commutant associated to the superalgebra $\AKMSA{pgl}{n}{n}$ at critical level \cite{CGL}.
The resulting commutant algebra was also found to be strongly and freely generated by $n+1$ fields, which we view as further evidence for Conjecture~\ref{conj:qh}.

For small $n$, the $\FS{n}$ algebras reduce to the $\beta\gamma$-ghosts for $n=1$, $\AKMA{sl}{2}_k$ in its universal form for $n=2$, and the Bershadsky-Polyakov algebra for $n=3$.  This last algebra is indeed known to be a quantum Hamiltonian reduction of $\AKMA{sl}{3}_k$ \cite{B,P} and its usual notation, $\FS{3}$, gives a rather more mundane explanation for
the notation chosen for the $\FS{n}$ algebras in general.  Recall that at non-generic levels, the universal vertex algebra associated to $\AKMA{sl}{n}_k$ ceases to be simple and one usually prefers to consider the simple quotient. Because of this, the algebras defined through hamiltonian reduction will not be simple for all levels and one should distinguish between them and their simple quotients.  

We quote what has been computed for the operator product expansions of $W^{(2)}_n$ at the level $k=-n^2/(n+1)$ that is of interest for this work.  These expansions are common to all non-trivial quotients of $\FS{n}$.  Because we will mostly concern ourselves with the connection to the singlet and triplet algebras $\cM(p)$ and $\cW(p)$, we will set $n$ throughout to $p-1$ for convenience.
\begin{prop}[Feigin-Semikhatov \cite{FS}] \label{prop:FSope}
Let $k=-(p-1)^2/p$, let $L$ be a Virasoro field of central charge $c=2-6(p-1)^2/p$, and let $H$, $E$ and $F$ be Virasoro primary fields
of conformal dimensions $1$, $(p-1)/2$ and $(p-1)/2$, respectively.  Then, the known part of the $\FS{p-1}$ operator product algebra at level $k$ includes
\begin{gather*}
H(z)H(w) \sim \frac{-2/p}{(z-w)^2} ,\qquad 
H(z)E(w) \sim \frac{E(w)}{(z-w)},\qquad
H(z)F(w) \sim  -\frac{F(w)}{(z-w)}, \\
E(z)E(w) \sim F(z)F(w) \sim 0, \\
\begin{split}
E(z)F(w) &= \frac{(-1)^{p}}{p^{p-1}}\frac{(2p-2)!}{(p-1)!}\frac{1}{(z-w)^{p-1}}+ 
 \frac{1}{2}\frac{(-1)^{p-1}}{p^{p-2}}\frac{(2p-2)!}{(p-1)!}\frac{H(w)}{(z-w)^{p-2}}+ \\
&\qquad \frac{1}{2}\frac{(-1)^{p}}{p^{p-3}}\frac{(2p-4)!}{(p-2)!}\frac{(p-2)\normord{H(w)H(w)}-\frac{2p-3}{p}\del H(w)-\frac{2}{p}L(w)}{(z-w)^{p-3}}+ \\
&\qquad \frac{(-1)^{p}}{p^{p-3}}\frac{(2p-4)!}{(p-1)!}\frac{1}{(z-w)^{p-4}}\Bigl(W(w)-\frac{(p-1)}{2p}\del L'(w)+\frac{(p-1)}{2}\normord{H(w)L'(w)}+\\
&\qquad (2p-3)(p-1)\bigl(-\frac{p}{24}\normord{H(w)H(w)H(w)}+\frac{1}{4}\normord{\del H(w)H(w)}-\frac{1}{6p}\del^2 H(w)\bigl)\Bigr)+\cdots,
\end{split}
\end{gather*}
where $L'=L+p\normord{HH}/4$ and $W$ is a dimension $3$ Virasoro primary.
The dots denote terms in which the exponent of $z-w$ is greater than $4-p$.
\end{prop}

\noindent We remark that the dimension three field $W$ only appears in the singular part of these operator product expansions when $p\geq 5$. In the case $p=5$, Feigin and Semikhatov also computed all the operator product expansions involving $W$. For a generic value of $k$, the resulting expressions are very long, but for $k=-(p-1)^2/p=-16/5$, they simplify considerably and are quoted for future reference.
 
\begin{prop}[Feigin-Semikhatov \cite{FS}] \label{prop:Wope}
When 
$p=5$ and $k=-16/5$, the operator product expansions for the dimension three field $W$ are
\begin{align*}
W(z)H(w) &\sim 0, \\
W(z)E(w) &\sim \frac{2}{3}\frac{\normord{H(w)\del E(w)}-2\normord{\del H(w)E(w)}+\frac{1}{5}\del^2E(w)-\normord{L(w)E(w)}}{(z-w)},\\
W(z)F(w) &\sim \frac{2}{3}\frac{\normord{H(w)\del F(w)}-2\normord{\del H(w)F(w)}-\frac{1}{5}\del^2F(w)+\normord{L(w)F(w)}}{(z-w)}, \\
W(z)W(w) &\sim \frac{16}{5} \frac{\Lambda(w)}{(z-w)^2} + \frac{8}{5} \frac{\del \Lambda(w)}{(z-w)}.
\end{align*}
Here, $\Lambda$ is a dimension $4$ Virasoro primary.  Its operator product expansion with $W$ involves descendants of $W$ and a Virasoro primary of dimension $5$ (see \cite[App. A.4.2]{FS}).
\end{prop}

\subsection{The triplet algebra as a coset}

As in \secref{sec:WnM}, take $\alpha_+=\sqrt{2p}$ and $\alpha_-=-\sqrt{2/p}$. 
We consider the lattices
\begin{equation}
D_+=\mathbb Z\alpha_+\beta_+, \qquad D_-=\mathbb Z\alpha_+\beta_-, \qquad D=\mathbb Z \frac{\alpha_+}{2}(\beta_+ + \beta_-)+ \mathbb Z \frac{\alpha_+}{2}(\beta_+ - \beta_-)\,,
\end{equation}
where $\beta_+$ and $\beta_-$ form a basis for a two-dimensional vector space over $\mathbb R$ with bilinear form chosen such that $\beta_+$ has length squared $1$, $\beta_-$ has length squared $-1$ and $\beta_+$ is orthogonal to $\beta_-$.
We define corresponding fields $\beta_+(z)$ and $\beta_-(z)$ with operator product expansions
\begin{equation}
\beta_+(z)\beta_+(w) \sim -\beta_-(z)\beta_-(w) \sim \log (z-w)\,, \qquad \beta_+(z) \beta_-(w) \sim 0.
\end{equation}
The derivatives of these fields define a rank $2$ Heisenberg vertex operator algebra $M$.
We then assert, in the usual fashion, the existence of lattice vertex operator algebras $V_{D_+}$, $V_{D_-}$ and $V_D$ associated to the respective lattices $D_+$, $D_-$ and $D$.
We also introduce screening charges similar to those considered in \secref{sec:Screening}:  $Q_\pm =\oint dz \normord{e^{\alpha_\pm \beta_+(z)}}$.
It is the kernel of $Q_-$ that forms our initial focus of attention.

\begin{prop} \label{prop:DefFields}
Define the following fields of $V_D$:
\begin{equation*}
\begin{split}
T(z) &= \frac{1}{2}\normord{\del\beta_+(z)\del\beta_+(z)}-\frac{1}{2}\normord{\del\beta_-(z)\del\beta_-(z)}+\frac{\alpha_++\alpha_-}{2}\del^2\beta_+(z),\\
T'(z) &= \frac{1}{2}\normord{\del\beta_+(z)\del\beta_+(z)}+\frac{\alpha_++\alpha_-}{2}\del^2\beta_+(z),\\
W^+(z) &= Q_+W^0(z),\qquad W^0(z) = Q_+W^-(z),\qquad W^-(z) = \normord{e^{-\alpha_+\beta_+(z)}}\,,\\
e(z)  &= \normord{e^{-\alpha_+(\beta_+(z)-\beta_-(z))/2}},\quad h(z) = \alpha_-\del\beta_-(z),\quad 
f(z) = \frac{(-1)^p (p-1)!}{p^{p-1}}Q_+\normord{e^{-\alpha_+(\beta_+(z)+\beta_-(z))/2}}.
\end{split}
\end{equation*}
Each of these fields belongs to 
the kernel of $Q_-$, denoted by $\ker_{V_D}(Q_-)$.  
\end{prop}
\begin{proof}
This is clear for $W^+$, $W^0$ and $W^-$, and the operator product expansion of $\normord{e^{\alpha_-\beta_+(z)}}$ with $e$, $h$ and $\normord{e^{-\alpha_+(\beta_++\beta_-)/2}}$ is easily checked to be regular. Thus, $e,h \in \ker_{V_D}(Q_-)$ too.
The statement also follows for $f$, since $Q_-$ and $Q_+$ commute. Finally, $T$ is in the kernel as $T'$ and $\normord{\del\beta_-\del\beta_-}$ are.
\end{proof}

\noindent It is well-known \cite{AM4} that $W^\pm$, $W^0$ and $T'$ strongly generate the triplet algebra $\cW(p)$.  It is also well-known \cite{A1} that $W^0$ and $T'$ strongly generate the singlet algebra $\cM(p)$. Furthermore, it is easily verified that $T$ is a Virasoro field of central charge $c=2-6(p-1)^2/p$.
For the module of $V_{D_-}$ corresponding to the coset $D_-+\alpha_+\beta_-\lambda/4$, we use
the convenient notation $\mathsf{V}_{[\lambda]}$. 

\begin{thm}
As a module of $\cW(p)\otimes V_{D_-}$,
\begin{equation*}
\ker_{V_D}(Q_-) \cong W_{1,1}\otimes \mathsf{V}_{[0]}   \oplus W_{2,1}\otimes \mathsf{V}_{[2]}.
\end{equation*}
Moreover, this kernel is a simple vertex operator algebra.  
\end{thm}
\begin{proof}
Using \eqref{eq:MinF} with $s=1$, we get, as a direct sum of irreducible singlet modules,
\begin{equation}
\begin{split}
\ker_{V_{D_+}}(Q_-) &\cong \bigoplus_{r\in\ZZ}M_{2r+1,1},\qquad
\ker_{V_{D_++\alpha_+/2}}(Q_-) \cong \bigoplus_{r\in\ZZ}M_{2r,1}
\end{split}
\end{equation}
and
by \eqref{eq:tripletdecomp}, as a direct sum of irreducible triplet representations,
\begin{equation}
\ker_{V_{D_+}}(Q_-) \cong W_{1,1}\qquad\text{and}\qquad     \ker_{V_{D_++\alpha_+/2}}(Q_-) \cong W_{2,1}.
\end{equation}
Obviously, $\text{ker}_{V_{D_-}}(Q_-)=V_{D_-}$ and hence
the kernel $\ker_{V_D}(Q_-)$ is a direct sum of two irreducible 
$\cW(p)\otimes V_{D_-}$-modules: 
\begin{equation}
\ker_{V_D}(Q_-) \cong W_{1,1}\otimes \mathsf{V}_{[0]} \oplus W_{2,1}\otimes \mathsf{V}_{[2]}.
\end{equation}
Finally, the modules $W_{1,1}$ and $\mathsf{V}_{[0]}$ are the identities in their respective 
fusion rings and $W_{2,1}$ and $\mathsf{V}_{[2]}$ are simple currents of order two:  $W_{2,1} \fuse W_{2,1} = W_{1,1}$ \cite{FHST,GR,TW1} and $\mathsf{V}_{[2]} \fuse \mathsf{V}_{[2]} = \mathsf{V}_{[0]}$.  We therefore obtain
\begin{equation}
\begin{split}
(W_{1,1}\otimes \mathsf{V}_{[0]}) \times (W_{1,1}\otimes \mathsf{V}_{[0]})  
&=  W_{1,1}\otimes \mathsf{V}_{[0]}, \\  
(W_{1,1}\otimes \mathsf{V}_{[0]}) \times (W_{2,1}\otimes \mathsf{V}_{[2]})  
&=  W_{2,1}\otimes \mathsf{V}_{[2]}, \\ 
(W_{2,1}\otimes \mathsf{V}_{[2]}) \times (W_{2,1}\otimes \mathsf{V}_{[2]})  
&=  W_{1,1}\otimes \mathsf{V}_{[0]}, 
\end{split}
\end{equation}
which together imply that the kernel is simple as a module of $\cW(p) \otimes V_{D_-}$, hence is a simple vertex operator algebra.
\end{proof}

\noindent It is now fairly simple to characterize $\cW(p)$ as a commutant inside the kernel of screenings.  
We recall that physicists refer to the commutant subalgebra as the coset algebra. 

\begin{thm}\label{thm:triplet}
$V_{D_-}$ and $\cW(p)$ form 
a mutually commuting pair inside $\text{ker}_D(Q_-)$. In other words, they form a Howe pair.
\end{thm}
\begin{proof}
An element in the commutant of a vertex algebra corresponds to an invariant state for that algebra, that is, a vacuum state. But the only vacuum state in the $\cW(p)$-module $W_{1,1}\oplus W_{2,1}$ is the highest weight state of $W_{1,1}$, hence the invariant states of $\text{ker}_{V_D}(Q_-)$, invariant under $\cW(p)$, have the form $1\otimes v_0$, $v_0 \in V_{[0]}$.  These clearly generate a copy of $V_{D_-}$.
Analogously, the only vacuum state for $V_{D_-}$ in $V_{[0]}\oplus V_{[2]}$ is that of $V_{[0]}$, hence the invariant states of $\text{ker}_{V_D}(Q_-)$, invariant under $V_{D_-}$, have the form $w_{1,1}\otimes 1$, $w_{1,1} \in W_{1,1}$. In this way, we get $\cW(p)$.
\end{proof}

\subsection{A vertex algebra homomorphism}

One of the main problems with explicitly working with the Feigin-Semikhatov algebras is that the defining operator product expansions are only known to a few orders (and what is known is already decidedly complex).  In this section, we construct a free field realisation which captures a reasonably large amount of this complexity.  More precisely, 
we show that there exists a surjective map from $\FS{n}$ to a subalgebra of the lattice vertex algebra $V_D$.  
This will be used in the next section to realise the singlet algebra $\cM(p)$ as a commutant subalgebra.

\begin{defn}
Denote by $\mathcal B_p$ the vertex operator subalgebra of $V_D$ generated by $e$, $h$, $f$ and $T$, as defined in Proposition \ref{prop:DefFields}.
\end{defn}

\noindent We wish to compare $\mathcal B_p$ with the Feigin-Semikhatov algebra $\FS{p-1}$ of level $k=-(p-1)^2/p$. For this, we compute some operator product expansions.  The calculations are straight-forward, but tedious, and are therefore omitted.

\begin{prop}\label{prop:Bope}
The field $T$ is Virasoro of central charge $c=2-6(p-1)^2/p$ in $\mathcal B_p$, while $h$, $e$ and $f$ are Virasoro primaries of conformal dimensions $1$, $n/2$ and $n/2$, respectively.  Moreover,
\begin{gather*}
h(z)h(w) \sim \frac{-2/p}{(z-w)^2} ,\qquad
h(z)e(w) \sim \frac{e(w)}{(z-w)},\qquad
h(z)f(w) \sim -\frac{f(w)}{(z-w)}, \\
e(z)e(w) \sim f(z)f(w) \sim 0,
\end{gather*}
and, if $p>2$,
\begin{align*}
e(z)f(w) &= \frac{(-1)^{p}}{p^{p-1}}\frac{(2p-2)!}{(p-1)!}\frac{1}{(z-w)^{p-1}}+ 
 \frac{1}{2}\frac{(-1)^{p-1}}{p^{p-2}}\frac{(2p-2)!}{(p-1)!}\frac{h(w)}{(z-w)^{p-2}}+ \\
&\qquad \frac{1}{2}\frac{(-1)^{p}}{p^{p-3}}\frac{(2p-4)!}{(p-2)!}\frac{(p-2)\normord{h(w)h(w)}-\frac{2p-3}{p}\del h(w)-\frac{2}{p}T(w)}{(z-w)^{p-3}}+ \\
&\qquad \frac{(-1)^{p}}{p^{p-3}}\frac{(2p-4)!}{(p-1)!}\frac{1}{(z-w)^{p-4}}\Bigl(-\frac{(p-1)}{2p}\del T'(w)+\frac{(p-1)}{2}\normord{h(w)T'(w)}+\\
&\qquad (2p-3)(p-1)\bigl(-\frac{p}{24}\normord{h(w)h(w)h(w)}+\frac{1}{4}\normord{\del h(w)h(w)}-\frac{1}{6p}\del^2 h(w)\bigl)\Bigr)+\cdots.
\end{align*}
If $p=2$, this latter operator product expansion is replaced by
\begin{align*}
e(z)f(w) &= \frac{1}{(z-w)} - h(w) - \bigl( \del h(w) + T(w) \bigr) (z-w) + \\
&\qquad 2 \left( \mathbb{W}(w) - \frac{1}{4} \del T'(w) + \frac{1}{2} \normord{h(w) T'(w)} - \right. \\
& \qquad \left. \frac{1}{12} \normord{h(w) h(w) h(w)} + \frac{1}{4} \normord{\del h(w) h(w)} - \frac{1}{12} \del^2 h(w) \right) (z-w)^2 + \cdots,
\end{align*}
where $\mathbb{W}$ is a dimension $3$ Virasoro primary.
Again, the dots denote higher-order terms. 
\end{prop}

\noindent Carefully comparing the operator product expansions of Propositions \ref{prop:FSope} and \ref{prop:Bope} now motivates the following definition:

\begin{defn} \label{def:omega}
For $p>2$, we define a map $\omega$ between the generators of $\FS{p-1}$ at level $k=-(p-1)^2/p$ and $\mathcal B_p$ as follows:
\[
\omega(H) = h, \quad \omega(E) = e,\quad\omega(F) = f,\quad\omega(L) = T,\quad \omega(W) = \omega(\Lambda) = \cdots = 0\,.
\]
Here, we let $W$, $\Lambda$, and the higher-dimensional Virasoro primaries they generate, be annihilated by $\omega$.  For $p=2$, we instead set
\[
\omega(H) = h, \quad \omega(E) = e,\quad\omega(F) = f,\quad\omega(L) = T,\quad \omega(W) = \mathbb{W},\quad \ldots \quad,
\]
where the dots indicate that one may identify non-zero fields in $\mathcal{B}_p$ which serve as the images under $\omega$ of the higher-dimensional Virasoro primaries.
\end{defn}

\noindent It appears that $\omega$ induces a surjective homomorphism of vertex operator algebras from $\FS{p-1}$, at the appropriate level, and $\mathcal B_p$.  Our lack of knowledge concerning the full structure of the Feigin-Semikhatov algebras for large $p$ makes this impossible to check in general.  However, we can verify it for $p \leq 5$.

First, note that the operator product expansion of $E$ and $F$ generates $H$, $L$, $W$, and probably the other higher-dimensional primaries.  Therefore, it is enough to verify the homomorphism property on $E$, $F$ and whichever fields appear in the singular terms of this expansion --- the strong generators --- because the fields appearing in the regular terms may be expressed as linear combinations of normally-ordered products of the strong generators.  (This enables one to compute, for example, the explicit form of the $p=2$ field $\mathbb{W}$ introduced in Proposition \ref{prop:Bope}.)

For $p=2$, the strong generators are just $E$ and $F$, so we need only compare
\[
E(z) F(w) \sim \frac{1}{(z-w)} \qquad \text{with} \qquad e(z) f(w) \sim \frac{1}{(z-w)}
\]
to guarantee that $\omega$ extends to a homomorphism.  Since both $h$ and $T$ (as well as $\mathbb{W}$) may be expressed as linear combinations of normally-ordered products of $e$ and $f$, this homomorphism is surjective.  The story is similar for $p=3$, for which $h$ becomes a strong generator, and $p=4$, for which $h$ and $T$ are both promoted to strong generators.

When $p=5$, $W$ becomes a strong generator, in addition to $h$ and $T$.  Thus, we need to check that its operator product expansions are consistent with $W \in \ker \omega$.  These were given in Proposition \ref{prop:Wope}.  We note that the expansion of $W$ with itself requires that $\Lambda \in \ker \omega$ and, in fact, that $W$ generates a 
proper ideal in the operator product algebra of $\FS{p-1}$.  Moreover, the expansions of $W$ with $E$ and $F$ require that the following non-trivial relations hold in $\mathcal B_5$:
\begin{lemma}\label{lemma:rel}
In $\mathcal B_5$, we have
\begin{equation*}
\begin{split}
0  &= \normord{h(z)\del e(z)}-2\normord{\del h(z)e(z)}+\frac{1}{5}\del^2e(z)-\normord{T(z)e(z)},\\
0  &=  \normord{h(z)\del f(z)}-2\normord{\del h(z)f(z)}-\frac{1}{5}\del^2f(z)+\normord{T(z)f(z)}.
\end{split}
\end{equation*}
\end{lemma}

\noindent Again, checking these relations is a straight-forward computation which will be omitted.  However, we mention that it is useful for these calculations to note that $f(z)$ has the following explicit form in $\mathcal B_5$:
\begin{multline*}
f(z) = -\frac{1}{5^4}:e^{\alpha_+(\beta_+(z)-\beta_-(z))/2}\Bigl(\alpha_+\del^4\beta_+(z)+4\alpha^2_+\del^3\beta_+(z)\del\beta_+(z)+3\alpha_+^2\del^2\beta_+(z)\del^2\beta_+(z)\\
+6\alpha_+^3\del^2\beta_+(z)\del\beta_+(z)\del\beta_+(z)+\alpha_+^4\del\beta_+(z)\del\beta_+(z)\del\beta_+(z)\del\beta_+(z)\Bigr):.
\end{multline*}
We remark that these relations mean that $\mathcal B_5$ is not freely generated by $e$, $f$, $h$ and $T$ (it is not universal as a vertex operator algebra).

To summarise, we have proven the following result:

\begin{thm} \label{thm:homo}
For $p \leq 5$, $\omega$ induces a surjective vertex algebra homomorphism between the Feigin-Semikhatov algebra $W^{(2)}_{p-1}$ of level $k=-(p-1)^2/p$ and $\mathcal B_p$.
\end{thm}

\noindent The following conjecture is therefore natural:

\begin{conj}
For general $n$, there exists a surjective vertex algebra homomorphism (extending $\omega$) between the Feigin-Semikhatov algebra $W^{(2)}_n$ of level $k=-\frac{n^2}{n+1}$ and $\mathcal B_{n+1}$.
\end{conj}

\noindent The obstruction to verifying this for higher $n$ is the unknown operator product expansions of the strong generators of dimension greater than $3$.

We remark that if the algebra considered by Feigin and Semikhatov in \cite{FS} turns out not to be freely generated, meaning that there are non-trivial linear dependencies among normally-ordered products of the generators, then an analogue of Theorem \ref{thm:homo} will still hold with $\FS{p-1}$ and $\mathcal B_p$ replaced by their appropriate quotients.  In particular, $\FS{1}$ may be identified with the $\beta \gamma$ ghost vertex algebra which is universal and simple.  The surjection $\omega$ therefore gives rise to an isomorphism of vertex algebras $\FS{1} \cong \mathcal B_2$.\footnote{In this respect, it is convenient that $p=2$ must be treated separately in Proposition \ref{prop:Bope}.  If $\mathbb{W}$ vanished (as it does for $p>2$), then $\omega$ would have a non-trivial kernel, contradicting the simplicity of $\FS{1}$.}  Similarly, $\FS{2}$ is the universal form of $\AKMA{sl}{2}$ at level $-\tfrac{4}{3}$ and one can easily check that the kernel generated by $W$ is the maximal ideal of $\FS{2}$ 
using the explicit knowledge of the singular vectors of the vacuum module.  In this case, $\omega$ induces an isomorphism between $\AKMA{sl}{2}_{-4/3}$ (the simple quotient of $\FS{2}$) and $\mathcal{B}_3$.  We will come back to these isomorphisms in Section \ref{sec:Branch}.

\subsection{The singlet as a commutant subalgebra}

We begin by recalling two results on the kernels of screenings due to Adamovi\'{c}:

\begin{lemma}[Adamovi\'{c} {\cite[Prop.~2.1 and Thm.~3.1]{A1}}] \label{thm:adam}
Let $M_+$ be the rank one Heisenberg vertex operator algebra generated by $\partial\beta_+$. 
Then, $\ker_{M_+}(Q_+)$ is the Virasoro algebra with Virasoro element $T'$ and 
$\ker_{M_+}(Q_-)$ is the singlet algebra $\cM(p)$ generated by $T'$ and $W^0$.
\end{lemma}

\noindent Recall that $h = \alpha_- \del \beta_-$ and that $M$ is the rank two Heisenberg vertex operator algebra generated by $\partial\beta_+$ and $\partial\beta_-$.  Using similar ideas to \cite{A2}, we show:

\begin{prop}
The singlet algebra $\cM(p)$ may be realised as a subalgebra of $\ker_{\mathcal B_p}(h_0)$:
\begin{equation}
 \ker_{\mathcal B_p}(h_0) = \ker_{M}(Q_-)=\ker_{M_+}(Q_-)\otimes M_-=\cM(p)\otimes M_-.
\end{equation}
In particular, $\cM(p)$ is a subalgebra of $\mathcal B_p$.
\end{prop}
\begin{proof}
Denote by $k_0$ the zero mode of $\partial\beta_+-\partial\beta_-$. 
Then $\mathcal B_p$ is in the kernel of $k_0$, since all its weak generators $e$, $f$, $h$ and $T$ are.
As before, let $M_\pm$ denote the rank one Heisenberg vertex operator algebra generated by $\partial\beta_\pm$.
Then $\ker_{V_D}(h_0)=V_{D_+}\otimes M_-$, while 
$\ker_{V_{D_+}\otimes M_-}(k_0)=M_+\otimes M_-$ and hence
 $\ker_{\mathcal B_p}(h_0)\subset M$ so that we have the inclusion $\ker_{\mathcal B_p}(h_0)\subset\ker_{M}(Q_-)$. 
For the other inclusion, we note that the second statement of Lemma \ref{thm:adam} implies that $\ker_{M}(Q_-)$ is generated by $T'$, $W^0$ and $\partial\beta_-$ and that the first
statement implies that $\ker_{M}(Q_+)$ is generated by $T'$ and $\partial\beta_-$.  As
$T'$ and $\partial\beta_-$ both have zero $h_0$-eigenvalue, it follows that $\ker_{M}(Q_+)\subset \ker_{\mathcal B_p}(h_0)$.
It remains to show that $W^0\in \ker_{\mathcal B_p}(h_0)$, and since $h_0W^0=0$, this means we have to show that $W^0 \in \mathcal B_p$.
Define\footnote{In what follows, we assume the mode expansion 
$e(z) = \sum_{n \in \ZZ} e_n z^{-n-1}$ familiar in the theory of vertex algebras, even when the conformal dimension of $e$ is not $1$.}
\[
v = \frac{2 (-1)^{p}p^{p-1}}{(p-1)!} e_{-p-1}f = 2  \normord{e^{-\alpha_+(\beta_+-\beta_-)/2}}_{-p-1}Q_+\normord{e^{-\alpha_+(\beta_++\beta_-)/2}} \ \in  \ \ker_{\mathcal B_p}(h_0)
\]
and $g_p(w)$ by $\partial_w^{(p-1)}e^{a(w)}=g_p(w)e^{a(w)}$. Then,
\[
Q_+\normord{e^{-\alpha_+(\beta_+\pm\beta_-)/2}}=\normord{\dfrac{g_p(\alpha_+\beta_+)e^{\alpha_+(\beta_+\mp\beta_-)/2}}{(p-1)!}}
\]
and thus 
\[
Q_+^2 \normord{e^{-\alpha_+(\beta_+\pm\beta_-)/2}} = 0.
\]
Hence,
\begin{equation*}
\begin{split}
Q_+v &= 2 \bigl(Q_+ \normord{e^{-\alpha_+(\beta_+-\beta_-)/2}}_{-p-1}\bigr) \bigl(Q_+\normord{e^{-\alpha_+(\beta_++\beta_-)/2}}\bigr) \\
&= Q_+^2\bigl( \normord{e^{-\alpha_+(\beta_+-\beta_-)/2}}_{-p-1}\normord{e^{-\alpha_+(\beta_++\beta_-)/2}}\bigr) \\
&= Q_+^2  \normord{e^{-\alpha_+\beta_+}} = Q_+ W^0\,.
\end{split}
\end{equation*}
It follows that $v-W^0 \in \ker_{M}(Q_+)\subset \ker_{\mathcal B_p}(h_0)$ and hence $W^0 \in \ker_{\mathcal B_p}(h_0)$.  
\end{proof}

\begin{thm} \label{thm:MCoset}
The singlet algebra $\cM(p)$ and the Heisenberg vertex algebra $M_-$ generated by $\del \beta_-$ are mutually commuting within $\mathcal B_p$.  In other words, they form a Howe pair.
\end{thm}
\begin{proof}
Let $\Com(A,B)$ denote the commutant algebra of $A$ as a subalgebra of $B$.
We first show that 
\begin{equation}\label{eq:coms}
\Com(\cM(p),\mathcal B_p) = \Com(\cM(p),\ker_{\mathcal B_p}(h_0))\qquad\text{and}\qquad \Com(M_-,\mathcal B_p) = \Com(M_-,\ker_{\mathcal B_p}(h_0))\,.
 \end{equation}
The second equality is obvious, since every element that commutes with $\partial\beta_-$ must be in the kernel of $h_0$.
The first equation is a little more involved.  For this, we note that every element that commutes with the singlet algebra must be in the kernel of the zero mode of the Virasoro field $T'$:
\[
\Com(\cM(p),\mathcal B_p)\subset \ker_{\mathcal B_p}(T'_0).
\]
We will show that $\ker_{\mathcal B_p}(T'_0) \subset \ker_{\mathcal B_p}(h_0)$, from which the first equation of \eqref{eq:coms} will follow immediately.

Let $V_n$ denote the $M$-module whose primary field is given by $v_n = \normord{e^{-n\alpha_+(\beta_+-\beta_-)/2}}$.  Then, the $V_n$ with $n \in \ZZ$ close under fusion, so we may conclude that
\[
V = \bigoplus_{n\in\mathbb Z} V_n
\]
defines a vertex operator subalgebra of $V_D$.  Moreover, $\mathcal B_p\subset V$ because $\mathcal B_p$ is generated by $e$ and $f$, both of which belong to $V$. 
We look for the fields of $V$ that are annihilated by $T'_0$.  The $T'_0$-eigenvalues of the $v_n$ are given by (see \eqref{eq:weightformula})
\[
\lambda_n = \frac{p}{4} n \left( n + \frac{2(p-1)}{p} \right).
\]
This is positive for $n \neq 0, -1$, hence any element of $V$ that is annihilated by $T'_0$ must live in either $V_0$ or $V_{-1}$.  In fact, $V_{-1}$ can be ruled out when restricting to $\mathcal{B}_p$ because the field of minimal $T'_0$-eigenvalue in $\mathcal{B}_p \cap V_{-1}$ is $f$ (its $T'_0$-eigenvalue is positive). 
We conclude that $\Com(\cM(p),\mathcal B_p)\subset V_0 \cap \mathcal B_p = \ker_{\mathcal B_p}(h_0)$, as required.

Finally, let $A$ and $B$ be two simple vertex operator algebras inside a third $C$, and suppose that $A$ commutes with $B$. Then $A$ and $B$ are a mutually commuting pair inside $A\otimes B$. This follows because an element
$a\otimes b\in A\otimes B$ commutes with $A$ if and only if $a$ is (a multiple of) the identity field $a(z)=I_A(z)$ on $A$ (and similarly for $B$). 
The claim of the theorem now follows from the identification $\ker_{\mathcal B_p}(h_0)=\cM(p)\otimes M_-$ and the fact that both $\cM(p)$ and $M_-$ are simple vertex algebras.
\end{proof}

\noindent We remark that the similar problem of looking for all operators of $\mathcal B_p$ that annihilate $M_-$ leads not only to all operators
of the singlet algebra but also includes the zero-mode $h_0$ of the operator algebra of $M_\beta$.

\section{Branching functions for $\mathcal B_2$ and $\mathcal B_3$} \label{sec:Branch}

This section is an application of the coset constructions resulting in 
Theorems \ref{thm:triplet} and \ref{thm:MCoset}. Irreducible modules
of $\ker_{V_D}(Q_-)$ and $\mathcal B_p$ decompose into modules 
of its mutually commuting subalgebras. Here, we will find these decompositions
at the level of characters when $p=2$ or $p=3$.

For this, we need to identify $\ker_{V_D}(Q_-)$ with certain extended algebras constructed
in \cite{CR1}. The construction of both these extended algebras and their modules relies on the conjecture that fusion respects spectral flow,
a conjecture that is consistent with the Verlinde formula for admissible level $\AKMA{sl}{2}$
\cite{CR1, CR4}. We will first outline some preliminary results concerning $\AKMA{sl}{2}$, formulating as well the conjecture that fusion respects spectral flow. Then, we perform the character
decompositions, using again results of \cite{CR1}.
 
\subsection{$\AKMA{sl}{2}$ at level $k$}

We first fix our notation and conventions for the affine vertex algebra corresponding to $\AKMA{sl}{2}$ at level $k$.
The affine Lie algebra $\AKMA{sl}{2}$ is generated by $h_n$, $e_n$, $f_n$ and $K$, for $n\in\mathbb Z$, with non-zero commutation 
relations\footnote{We follow \cite{R3} here in choosing a basis of $\SLA{sl}{2}$ which is adapted to a triangular decomposition that respects the adjoint defining the real form $\SLA{sl}{2 ; \RR}$.  There is a subtlety here in that the adjoint of the chiral algebra must extend to an adjoint on the simple current extended algebra which must be consistent with the mutual localities of the chiral and extension fields.  For $k=-\tfrac{1}{2}$, it was shown in \cite{R3} that choosing the $\SLA{su}{2}$ adjoint leads to a non-associative extended algebra whereas choosing the $\SLA{sl}{2;\RR}$ adjoint leads to the algebra of the $\beta \gamma$ ghosts.}
\begin{equation}
\comm{h_m}{e_n} = 2 e_{m+n}, \ \ \comm{h_m}{h_n} = 2m \delta_{m+n,0} K, \ \ \comm{e_m}{f_n} = -h_{m+n} - m \delta_{m+n,0} K, \ \ \comm{h_m}{f_n} = -2 f_{m+n}.
\end{equation}
We fix $K$ to act as multiplication by a fixed number $k$, called the level, on modules.  The conformal structure for $k \neq -2$ is given by the standard Sugawara construction:
\begin{equation}
L_n = \frac{1}{2 \brac{k+2}} \sum_{r \in \ZZ} \normord{\frac{1}{2} h_r h_{n-r} - e_r f_{n-r} - f_r e_{n-r}}.
\end{equation}
The central charge is $c = 3k/(k+2)$.  Of course, the $h_n$ generate a copy of the Heisenberg algebra $\AKMA{gl}{1}$.

Recall the family of spectral flow automorphisms, parameterized by $s\in\mathbb Z$, for $\AKMA{sl}{2}$ at level $k$:
\begin{equation}
\sigma_s(h_n) = h_n-\delta_{n,0}sk,\qquad
\sigma_s(e_n) = e_{n-s},\qquad  
\sigma_s(f_n) = f_{n+s},\qquad  
\sigma_s(L_0) = L_0-\frac{s}{2}h_0+\frac{s^2}{4}k.
\end{equation}
When $V$ is a level $k$ $\AKMA{sl}{2}$-module, we may define another level $k$ $\AKMA{sl}{2}$-module $V^s$ as follows.
 To any element $v\in V$ we associate the twisted element $\sigma_s^*(v)$. As a set, $V^s$ is given by all these twisted elements
and is isomorphic to $V$ as a vector space.  
Then, the action of $X\in\AKMA{sl}{2}$ on $V^s$ is defined by
\begin{equation}
X\sigma_s^*(v) = \sigma_s^*(\sigma^{-1}_s(X)v)\,.
\end{equation}
Note that $V^s$ and $V$ are not usually isomorphic as $\AKMA{sl}{2}$-modules (unless $s=0$).  We thus find
\begin{equation}
h_0\sigma_s^*(\ket{0}) = sk \sigma_s^*(\ket{0}) \quad\text{and}\quad 
L_0\sigma_s^*(\ket{0}) = \frac{s^2}{4}k \sigma_s^*(\ket{0}),
\end{equation}
where $\ket{0}$ denotes the $\AKMA{sl}{2}$ vacuum state.

We are looking for states that will correspond to the generators $W^+$, $W^0$, $W^-$ of $\cW(p)$.  For $p=2$, the appropriate $\AKMA{sl}{2}$ level is $k=-1/2$ and for $p=3$, it is $k=-4/3$.

\begin{prop}
When $k=-1/2$, the states
$W_2^+=\sigma^*_4(e_{-1}\ket{0})$ and $W_2^- =\sigma^*_{-4}(f_{-1}\ket{0})$ both then have conformal dimension three and they are vacuum states for the $\AKMA{gl}{1}$-subalgebra generated by $h$.

\noindent When $k=-4/3$, 
$W_3^+=\sigma^*_3(e_{-1}e_{-1}\ket{0})$ and $W_3^-=\sigma^*_{-3}(f_{-1}f_{-1}\ket{0})$
are vacuum states for the 
$\AKMA{gl}{1}$-subalgebra generated by $h$ of conformal dimension five.
\end{prop}
\begin{proof}
$W_2^+$ is invariant under $\AKMA{gl}{1}$ because
\begin{equation}
h_{0}W^+ = \sigma^*_4((h_0 + 4k) e_{-1} \ket{0}) = (4k+2)W^+ = 0
\end{equation}
and, obviously, $h_nW^+=0$ for $n>0$.
Further, its conformal dimension is $3$ because
\begin{equation}
L_0W^+ = \sigma^*_4((L_0 + 2h_0 + 4k) e_{-1} \ket{0}) = \Bigl(1 + 4 + 4k\Bigr)W^+ = 3W^+\, .
\end{equation}
An analogous argument shows that $W_3^-$ is invariant under $\AKMA{gl}{1}$ and likewise has conformal dimension $3$. 
The argument for the case $k=-4/3$ is identical.
\end{proof}

An important conjecture for the representation theory of affine vertex algebras is
that fusion rules are compatible with spectral flow automorphisms. In the case of interest to us, this is the following:

\begin{conj} \label{conj:SF}
Let $V$ and $W$ be (admissible) level $k$ modules of $\AKMA{sl}{2}$, where $k\in\{-1/2,-4/3\}$. Then,
\begin{equation}
V^s\times W^t=(V\times W)^{s+t}.
\end{equation}
\end{conj}

\noindent The fusion rules at these levels have been partially computed in \cite{G,R1} and a Verlinde formula for the Grothendieck ring of characters has been evaluated in \cite{CR1}. In both instances, the results strongly support the conjecture, as do the results known for more general admissible levels \cite{CR4}.

We denote the vacuum module at level $k$ by $\IrrMod{0}$.
Since $W^\pm_2\in \IrrMod{0}^{\pm 4}$, we are interested in an extension generated by these
two modules. 
Assuming Conjecture \ref{conj:SF}, the modules $\IrrMod{0}^{\pm 4}$ generate a free abelian group of simple currents.  Combining the fusion orbit of these simple currents on the vacuum module $\IrrMod{0}$ of level $k=-1/2$, we obtain the module
\begin{equation}
\mathcal A_2 = \bigoplus_{m\in\mathbb Z}\IrrMod{0}^{4m}
\end{equation}
which constitutes the vacuum module of a simple current extension of $\AKMA{sl}{2}_{-1/2}$ \cite{CR1}.  We will also denote the corresponding extended algebra by $\mathcal A_2$.  The fusion orbits through the other $\AKMA{sl}{2}_{-1/2}$-modules similarly define (untwisted) $\mathcal A_2$-modules when the conformal dimensions of the states in the given orbit all differ by integers.  
Similarly, for $k=-4/3$, the module 
\begin{equation}
\mathcal A_3 = \bigoplus_{m\in\mathbb Z}\IrrMod{0}^{3m}
\end{equation}
is the module of a simple current extension of $\AKMA{sl}{2}_{-4/3}$
and we also call the corresponding extended algebra $\mathcal A_3$.
This extended algebra has been constructed in \cite{A3}.

Finally, recall \cite{AM3,G,R1} that for $k=-1/2$ or $-4/3$, hence $c=-1$ or $-6$ (respectively), 
there is a family of representations $\TypMod{\lambda}^s$ of $\AKMA{sl}{2}_k$, called the \emph{standard modules}\footnote{The name standard refers to the fact that these
modules provide the generic family of representations.} in \cite{CR5}, 
which are labelled by a weight ($h_0$-eigenvalue) $\lambda \in \RR / 2 \ZZ$
and a spectral flow index $s \in \ZZ$.
For \(s=0\) the $\TypMod{\lambda}$ are examples of \emph{relaxed} highest weight
modules \cite{FSST}. They are affinisations of certain 
$\SLA{sl}{2}$-modules that are neither highest nor lowest weight, but instead have weights $\lambda + 2n$, $n \in \ZZ$, each with multiplicity one.
Generically, the standard modules are irreducible and their characters are given by
\begin{equation} \label{ch:TypSL2}
\ch{\TypMod{\lambda}^s} = \tr_{\TypMod{\lambda}^s} z^{h_0} q^{L_0 - c/24} = \frac{1}{\func{\eta}{q}^2} \sum_{n \in \ZZ} z^{2n + \lambda + ks} q^{s(2n + \lambda + ks/2)/2}.
\end{equation}
The non-generic case corresponds to $\lambda = \pm k \bmod{2}$, in which case the character formula \eqref{ch:TypSL2} still holds, but the modules become reducible but
indecomposable.  The irreducible quotient modules at these non-generic parameters have characters which may be expressed as infinite (but convergent) linear combinations of the non-generic $\TypMod{\lambda}^s$ characters.  We will detail this in the following as necessary.
Modules of the extended algebras are obtained by combining the appropriate spectral flow orbits.

\subsection{Branching functions for $\mathcal B_2$ and $\AKMA{sl}{2}$ at level $k=-1/2$}

In this subsection, we consider the case $p=2$ and its relation to $\AKMA{sl}{2}$ at level $k=-1/2$;
for a review of the representation theory of the latter, see \cite{CR5}.
The standard modules $\TypMod{\lambda}^s$, 
 with spectral flow index $s\in\ZZ$ and weight label $\lambda\in\RR/2\ZZ$ are
 irreducible for $\lambda\neq\pm 1/2\bmod{2}$.
There are also non-standard irreducible modules $\IrrMod{\mu}^s$, with $\mu \in \set{0,1}$ and $s \in \ZZ$.  When $s=0$ and $s=1$, the non-standard irreducibles are highest weight modules.  $\IrrMod{0}$ is the vacuum module.

When $\lambda=\pm 1/2$, the standard modules are indecomposable of length two with non-standard irreducibles for composition factors:
\begin{equation}
\begin{split}
0 \longrightarrow \IrrMod{1}^{s+1} \longrightarrow \TypMod{1/2,+}^s 
\longrightarrow \IrrMod{0}^{s-1}  \longrightarrow 0,\qquad
0 \longrightarrow \IrrMod{0}^{s+1} \longrightarrow \TypMod{-1/2,+}^s 
\longrightarrow \IrrMod{1}^{s-1}  \longrightarrow 0.
\end{split}
\end{equation}
Here the subindex $+$ indicates that $\TypMod{\pm 1/2,+}$
possesses a highest weight vector. The conjugate modules $\TypMod{\mp 1/2,-}$ are also indecomposable and are described by similar short exact sequences. 

\subsubsection{Branching functions and the singlet theory}

We now verify explicitly that the $\cM(2)$-characters are precisely the branching functions obtained by decomposing irreducible $\AKMA{sl}{2}_{-1/2}$-characters into irreducible $\AKMA{gl}{1}$-characters.  Theorem \ref{thm:MCoset} in fact guarantees that we can do this for $\FS{1}$-characters.  However, $\FS{1}$ is the $\beta \gamma$ ghost system which is the order $2$ simple current extension of $\AKMA{sl}{2}_{-1/2}$ by $\IrrMod{1}$ \cite{R3}, an irreducible module whose weights ($h_0$-eigenvalue) are odd.  It follows that any $\FS{1}$-module decomposes into a direct sum of two $\AKMA{sl}{2}_{-1/2}$-modules whose weights (mod $2$) differ by $1$.  As $\AKMA{gl}{1}$-modules have constant weight, we may conclude that the branching functions for the decomposition of $\AKMA{sl}{2}_{-1/2}$-characters into $\AKMA{gl}{1}$-characters will indeed be $\cM(2)$-characters.

To check this, we first note that the $\AKMA{gl}{1}$ subalgebra generated by $h$ has lorentzian signature. Denoting its irreducible modules by $\mathsf{F}_{\lambda}$, where $\lambda$ is its common weight, the characters have the form
\begin{equation}
 \ch{\mathsf{F}_{\lambda}} = \frac{z^\lambda q^{-\lambda^2/2}}{\func{\eta}{q}},
\end{equation}
with $z$ keeping track of the weight and $q$ the conformal dimension.
Next, we recall from Section \ref{sec:Screening} that there are $\cM(2)$-modules $\mathcal F_{\mu}$ (of central charge $-2$) which are generically irreducible, generic now meaning that $\mu \notin \ZZ$, whose characters have the form
\begin{equation} \label{ch:TypSing}
\ch{\mathcal F_{\mu}} = \frac{q^{(\mu - \alpha_0/2)^2/2}}{\func{\eta}{q}}.
\end{equation}
Here, we recall that $\alpha_0 = 1$ for $p=2$.  We now have the following character decomposition, realising the generic singlet characters as branching functions of the standard $\AKMA{sl}{2}_{-1/2}$ characters:
\begin{equation} \label{eqn:TypBranch1}
\ch{\TypMod{\lambda}^s} = \sum_{n \in \ZZ} \frac{z^{2n+\lambda-s/2} q^{-(2n+\lambda-s/2)^2/2}}{\func{\eta}{q}} \cdot \frac{q^{(2n+\lambda)^2/2}}{\func{\eta}{q}} = \sum_{n \in \ZZ}\ch{\mathsf{F}_{2n+\lambda-s/2}}\cdot \ch{\mathcal F_{2n+\lambda+1/2}}.
\end{equation}
It is worth remarking that the branching functions (the $\cM(2)$-characters) do not depend upon the spectral flow index $s$.

The computations for the non-generic irreducible singlet characters may be detailed explicitly, but they follow more easily (and more elegantly) from the simple derivation above by noting that the non-generic irreducible characters may be written as infinite linear combinations of standard characters \eqref{ch:TypSL2} \cite{CR1}:
\begin{equation}
\ch{\IrrMod{\lambda}^s} = \sum_{\ell = 0}^{\infty} \brac{-1}^{\ell} \ch{\TypMod{\lambda + \ell + 1/2}^{2 \ell + s + 1}} \qquad \text{($\lambda = 0,1$).}
\end{equation}
Applying \eqref{eqn:TypBranch1}, we therefore obtain
\begin{equation}
\ch{\IrrMod{\lambda}^s} = \sum_{n \in \ZZ} \ch{\mathsf{F}_{2n+\lambda-s/2}} \sum_{\ell = 0}^{\infty} \brac{-1}^{\ell} \ch{\mathcal F_{2n+\lambda+\ell+1}}
\end{equation}
in which we recognise, using \eqref{eq:singletchars} and \eqref{ch:TypSing}, the sum over $\ell$ as a non-generic $\cM(2)$-character:
\begin{equation} \label{ch:p=2Irr}
\ch{\IrrMod{\lambda}^s} = \sum_{n \in \ZZ} \ch{\mathsf{F}_{2n+\lambda-s/2}} \cdot \ch{M_{2n+\lambda+1 , 1}}.
\end{equation}
We summarize this as

\begin{prop} The characters of the irreducible $\AKMA{sl}{2}$-modules at level $k=-1/2$ decompose into
$\AKMA{gl}{1}$ and $\cM(2)$ characters as follows:
\begin{equation}
\ch{\TypMod{\lambda}^s} = \sum_{n \in \ZZ}\ch{\mathsf{F}_{2n+\lambda-s/2}}\cdot \ch{\mathcal F_{2n+\lambda+1/2}},\qquad
\ch{\IrrMod{\lambda}^s} = \sum_{n \in \ZZ} \ch{\mathsf{F}_{2n+\lambda-s/2}} \cdot \ch{M_{2n+\lambda+1 , 1}}.
\end{equation}
The first decomposition also describes that of the reducible standard $\AKMA{sl}{2}$-modules.
\end{prop}

$\mathcal B_2$, the $\beta\gamma$ ghost system, being a simple current extension of $\AKMA{sl}{2}$
at level $k=-1/2$, has modules $\mathbb L_0^s$ and $\mathbb E_\lambda^s$, for 
$\lambda \in \RR / \ZZ$ and $s \in \ZZ$.  The latter are irreducible unless $\lambda = 1/2 \bmod{2}$.  As $\AKMA{sl}{2}$-modules, they decompose as
\begin{equation}
\mathbb L_0^s=\IrrMod{0}^s\oplus \IrrMod{1}^s,\qquad
\mathbb E_\lambda^s= \TypMod{\lambda}^s\oplus\TypMod{\lambda+1}^s,
\end{equation}
hence we obtain:

\begin{prop} The characters of the irreducible $\mathcal B_2$-modules decompose into
$\AKMA{gl}{1}$ and $\cM(2)$ characters as follows:
\begin{equation}
\ch{\mathbb E_\lambda^s} = \sum_{n \in \ZZ}\ch{\mathsf{F}_{n+\lambda-s/2}}\cdot \ch{\mathcal F_{n+\lambda+1/2}},\qquad
\ch{\mathbb L_0^s} = \sum_{n \in \ZZ} \ch{\mathsf{F}_{n-s/2}} \cdot \ch{M_{n+1 , 1}}.
\end{equation}
The first decomposition also describes that of the reducible $\mathcal B_2$-modules $\mathbb E_{1/2}^s$.
\end{prop}

\noindent We remark that such a module decomposition was guaranteed by Theorem \ref{thm:MCoset}.

\subsubsection{Branching functions and the triplet theory}

We have seen that the candidate states for the triplet generators $W^{\pm}$ do not correspond to fields of the affine algebra but instead belong to an extended algebra that we have denoted by $\mathcal A_2$. The fusion orbits through the $\TypMod{\lambda}^s$ give rise to (untwisted) extended algebra modules when the weight $\lambda$ has the form $j/2$, for $j \in \ZZ$.  We therefore obtain a family parametrised by $j=0,1,2,3$ and a spectral flow index $r=0,1,2,3$ \cite{CR1}.  When $j$ is even, the resulting $\mathcal A_2$-module is irreducible, whereas those with $j$ odd are reducible but indecomposable.  Their characters take the form
\begin{equation} \label{eqn:p=2TypDecomp}
 \sum_{s\in 4\ZZ} \ch{\TypMod{j/2}^{r+s}} = \sum_{m\in\ZZ}  \ch{\mathsf{F}_{2m+(j-r)/2}}\sum_{n\in\ZZ}
\ch{\mathcal F_{2n+(j+1)/2}},
\end{equation}
from which we observe that the first sum gives characters of the lattice vertex algebra $V_{D_-}$,
\begin{equation}
\sum_{m\in\ZZ}  \ch{\mathsf{F}_{2m+(j-r)/2}} = \ch{\mathsf{V}_{[(j-r)/2]}},
\end{equation}
corresponding to the coset $D_-+\beta_-(j-r)/4$.
The second sum (the branching functions) are the following $\cW(2)$-characters:
\begin{equation}
\begin{split}
 \sum_{n\in\ZZ}
\ch{\mathcal F_{2n+1/2}}&= \ch{\mathcal{V}_{[\alpha_{1,2}]}} = \ch{W_{1,2}}, \\
\sum_{n\in\ZZ}
\ch{\mathcal F_{2n+1}}&= \ch{\mathcal{V}_{[\alpha_{2,1}]}} = \ch{W_{1,1}} + \ch{W_{2,1}},\\
\sum_{n\in\ZZ}
\ch{\mathcal F_{2n+3/2}}&=\ch{\mathcal{V}_{[\alpha_{2,2}]}} = \ch{W_{2,2}}, \\
\sum_{n\in\ZZ}
\ch{\mathcal F_{2n+2}}&= \ch{\mathcal{V}_{[\alpha_{1,1}]}} = \ch{W_{1,1}} + \ch{W_{2,1}}.                          
\end{split}
\end{equation}
This demonstrates that the $\mathcal{A}_2$-characters built from the $\TypMod{j/2}^s$ decompose as a $V_{D_-}$-character times a $\cW(2)$-character.  Similarly, the decomposition for the non-standard irreducibles follows immediately from \eqref{ch:p=2Irr}:
\begin{equation}
\sum_{s \in 4 \ZZ} \ch{\IrrMod{\lambda}^{s+r}} = \sum_{m \in \ZZ} \ch{\mathsf{F}_{2m+\lambda-r/2}} \cdot \sum_{n \in \ZZ} \ch{\mathcal F_{2n+\lambda+1 , 1}}
=  \ch{\mathsf{V}_{[\lambda-r/2]}} \cdot \ch{W_{\lambda +1,1}}.
\end{equation}
Summarizing, we get:

\begin{prop}
Characters of irreducible $\mathcal A_2$-modules decompose into 
$V_{D_-}\otimes \AKMA{sl}{2}_{-1/2}$-characters as
\begin{equation}
\begin{split}
\sum_{s\in 4\ZZ} \ch{\TypMod{j/2}^{r+s}} &=\ch{\mathsf{V}_{[-r/2]}}\cdot\ch{W_{1,2}},\qquad
\sum_{s\in 4\ZZ} \ch{\TypMod{j/2}^{r+s}} =\ch{\mathsf{V}_{[(2-r)/2]}}\cdot\ch{W_{2,2}},\\
\sum_{s \in 4 \ZZ} \ch{\IrrMod{0}^{s+r}} &= \ch{\mathsf{V}_{[-r/2]}}\cdot \ch{W_{1,1}},\qquad
\sum_{s \in 4 \ZZ} \ch{\IrrMod{1}^{s+r}} = \ch{\mathsf{V}_{[1-r/2]}}\cdot \ch{W_{2,1}}.
\end{split}
\end{equation}
\end{prop}

\noindent In particular, we have
\begin{equation}
\ch{\mathcal A_2}=\sum_{s \in 4 \ZZ} \ch{\IrrMod{0}^{s}} = \ch{\mathsf{V}_{[0]}} \cdot \ch{W_{1,1}}.
\end{equation}

It is now straightforward to lift this analysis to the extension $\mathbb A_2$ of $\mathcal B_2$ by
the simple current $\mathbb L_0^4$:
\begin{equation}
\mathbb A_2=\bigoplus_{s\in\ZZ}\mathbb L_0^{4s}.
\end{equation}

\begin{prop}
Characters of irreducible $\mathbb A_2$-modules decompose into 
$V_{D_-}\otimes \mathcal B_2$-characters as
\begin{equation}
\begin{split}
\sum_{s\in 4\ZZ} \ch{\mathbb E_0^{r+s}} &=
\ch{\mathsf{V}_{[1-r/2]}}\cdot \ch{W_{2,2}}+
\ch{\mathsf{V}_{[-r/2]}}\cdot \ch{W_{1,2}}, \\
\sum_{s\in 4\ZZ} \ch{\mathbb L_0^{r+s}} &=
\ch{\mathsf{V}_{[1-r/2]}}\cdot \ch{W_{2,1}}+
\ch{\mathsf{V}_{[-r/2]}}\cdot \ch{W_{1,1}}. 
\end{split}
\end{equation}
Here $r\in\{0,1,2,3\}$.
\end{prop}

\noindent We remark that the special case
\begin{equation}
\ch{\mathbb A_2}=\sum_{s \in 4 \ZZ} \ch{\mathbb L_0^{s}}= \ch{\mathsf{V}_{[0]}} \cdot \ch{W_{1,1}}+\ch{\mathsf{V}_{[1]}}\cdot \ch{W_{2,1}}=\ch{\ker_{V_D}(Q_-)}
\end{equation}
is consistent with Theorem \ref{thm:triplet}, so it is natural to conjecture that $\mathbb A_2\cong \ker_{V_D}(Q_-)$.

\subsection{Branching functions for $\mathcal B_3$}

We now turn to the case $p=3$ and $\AKMA{sl}{2}$ of level $k=-4/3$.
As with $k=-1/2$, the irreducible modules may be described as being standard or non-standard. 
The irreducible standard modules $\TypMod{\lambda}^s$ 
again have spectral flow index $s\in\ZZ$ and weight label $\lambda\in\RR/ 2\ZZ$, but now we require that 
$\lambda\neq\pm 2/3\bmod{2}$.  The non-standard irreducibles fall into two families 
$\IrrMod{0}^s$ and $\IrrMod{-2/3}^s$.
When $\lambda=\pm 2/3 \bmod{2}$, the standard modules have the following non-split short exact sequences:
\begin{equation}
0 \longrightarrow \IrrMod{0}^{s+1} \longrightarrow \TypMod{2/3,+}^s 
\longrightarrow \IrrMod{-2/3}^{s-1}  \longrightarrow 0,\qquad
0 \longrightarrow \IrrMod{-2/3}^{s+1} \longrightarrow \TypMod{-2/3,+}^s 
\longrightarrow \IrrMod{0}^{s-1}  \longrightarrow 0.
\end{equation}
As before, the subindex $+$ indicates that $\TypMod{\pm 2/3,+}$
possesses a highest weight vector. Its conjugate module is denoted by $\TypMod{\mp 2/3,-}$.

\subsubsection{Branching functions and the singlet theory}

We now verify explicitly that the $\cM(3)$-characters are precisely the branching functions obtained by decomposing irreducible $\AKMA{sl}{2}_{-4/3}$-characters into irreducible $\AKMA{gl}{1}$-characters.  
Denoting the irreducible $\AKMA{gl}{1}$-modules by $\mathsf{F}_{\lambda}$, where $\lambda$ is the common weight, the characters are
\begin{equation}
 \ch{\mathsf{F}_{\lambda}} = \frac{z^\lambda q^{-3\lambda^2/16}}{\func{\eta}{q}}.
\end{equation}
Moreover, we recall from Section \ref{sec:Screening} that the $\cM(3)$-modules 
$\mathcal F_{\mu}$ (of central charge $c=-7$) are irreducible when $\mu \notin \ZZ$ and that their characters have the form
\begin{equation} \label{ch:TypSing2}
\ch{\mathcal F_{\mu}} = \frac{q^{(\mu - \alpha_0/2)^2/2}}{\func{\eta}{q}},
\end{equation}
where $\alpha_0 = \sqrt{8/3}$ for $p=3$.  The character decomposition realising the generic singlet characters as branching functions is then
\begin{equation} \label{eqn:TypBranch2}
\begin{split}
\ch{\TypMod{\lambda}^s} &= \sum_{n \in \ZZ} \frac{z^{2n+\lambda-4s/3} 
q^{-3(2n+\lambda-4s/3)^2/16}}{\func{\eta}{q}} 
\cdot \frac{q^{3(2n+\lambda)^2/16}}{\func{\eta}{q}}\\
 &= \sum_{n \in \ZZ}\ch{\mathsf{F}_{2n+\lambda-4s/3}}
\cdot \ch{\mathcal F_{(2n+\lambda)/\alpha_0+\alpha_0/2}}.
\end{split}
\end{equation}

For the non-generic irreducible characters, there are again expressions in terms of infinite linear combinations of characters of the forms \eqref{ch:TypSL2} \cite{CR1}:
\begin{equation}
\begin{split}
\ch{\IrrMod{0}^s} = \sum_{\ell = 0}^{\infty} 
\brac{\ch{\TypMod{-2/3}^{3 \ell + s + 1}}-\ch{\TypMod{2/3}^{3 \ell + s + 2}}},\qquad
\ch{\IrrMod{-2/3}^s} = \sum_{\ell = 0}^{\infty} 
\brac{\ch{\TypMod{2/3}^{3 \ell + s + 1}}-\ch{\TypMod{-2/3}^{3 \ell + s + 3}}}.
\end{split}
\end{equation}
Applying \eqref{eqn:TypBranch2}, we therefore obtain
\begin{equation}\label{ch:p=3Irr}
\begin{split}
\ch{\IrrMod{0}^s} &= \sum_{n \in \ZZ} \ch{\mathsf{F}_{2n-2-4s/3}} \sum_{\ell = 0}^{\infty} 
\brac{\ch{\mathcal F_{(2n+4\ell-2/3)/\alpha_0+\alpha_0/2}}-
\ch{\mathcal F_{(2n+4\ell+2/3)/\alpha_0+\alpha_0/2}}},\\
\ch{\IrrMod{-2/3}^s} &= \sum_{n \in \ZZ} \ch{\mathsf{F}_{2n-8/3-4s/3}} 
\sum_{\ell = 0}^{\infty} \brac{\ch{\mathcal F_{(2n+4\ell-4/3)/\alpha_0+\alpha_0/2}}-
\ch{\mathcal F_{(2n+4\ell+4/3)/\alpha_0+\alpha_0/2}}}.\\
\end{split}
\end{equation}
Simplifying, we arrive at:

\begin{prop} Characters of $\AKMA{sl}{2}$-modules at level $k=-4/3$ decompose into
$\AKMA{gl}{1}$ and $\cM(3)$ characters as follows:
\begin{equation}
\begin{split}
\ch{\TypMod{\lambda}^s} &= \sum_{n \in \ZZ}\ch{\mathsf{F}_{2n+\lambda-4s/3}}
\cdot \ch{\mathcal F_{(2n+\lambda)/\alpha_0+\alpha_0/2}},\\
\ch{\IrrMod{0}^s} &= \sum_{n \in \ZZ} \ch{\mathsf{F}_{2n-2-4s/3}}\cdot\ch{M_{n,1}},\qquad
\ch{\IrrMod{-2/3}^s} = \sum_{n \in \ZZ} \ch{\mathsf{F}_{2n-8/3-4s/3}}\cdot\ch{M_{n,2}}.
\end{split}
\end{equation}
\end{prop}

\noindent Note that this again reflects Theorem \ref{thm:MCoset}.

\subsubsection{Branching functions and the triplet theory}

The candidates for the triplet generators $W^{\pm}$ have been identified as belonging to the extended algebra $\mathcal A_3$. 
This time, the fusion orbits through the standard modules $\TypMod{\lambda}^s$ of $\AKMA{sl}{2}_{-4/3}$
give rise to (untwisted) extended algebra modules if the weight $\lambda$ is in $\{0,\pm 2/3\}$.
We therefore obtain a family of extended modules parametrised by $\lambda=0,\pm 2/3$ and the spectral flow index $r=0,1,2$ \cite{CR1}.  When $\lambda=0$, the resulting $\mathcal A_3$-module is irreducible; otherwise, they are reducible but indecomposable.  Their characters take the form
\begin{equation} \label{eqn:p=2TypDecomp'}
\begin{split}
 \sum_{s\in 3\ZZ} \ch{\TypMod{\lambda}^{s+r}} &= \sum_{n,s\in\ZZ}  \ch{\mathsf{F}_{2n+\lambda -4r/3-4s}} \cdot
\ch{\mathcal F_{(2n+\lambda)/\alpha_0+\alpha_0/2}}\\
&= \ch{\mathsf{V}_{[\lambda-4r/3]}}\cdot\ch{V_{[\alpha_{2,3\lambda/2}]}}+        \ch{\mathsf{V}_{[\lambda+2-4r/3]}}\cdot\ch{V_{[\alpha_{1,3\lambda/2}]}}
\end{split}
\end{equation}
Here, as in the last section, the $\mathsf{V}_{[\lambda-4r/3]}$
are the modules of the lattice vertex algebra $V_{D_-}$ corresponding to the coset $D_-+\beta_-\alpha_+(\lambda-4r/3)/4$:
\begin{equation}
\sum_{m\in\ZZ}  \ch{\mathsf{F}_{4m+\lambda-4r/3}} = \ch{\mathsf{V}_{[\lambda-4r/3]}}.
\end{equation}
The branching functions are the $\cW(3)$-characters
\begin{equation}
\sum_{n\in\ZZ}
\ch{\mathcal F_{(4n+\lambda)/\alpha_0+\alpha_0/2}}=\sum_{n\in\ZZ}
\ch{\mathcal F_{\alpha_{2n,3\lambda/2}}}=\ch{V_{[\alpha_{1,3\lambda_2}]}}
\end{equation}
which are, in terms of irreducible $\cW(3)$-characters,
\begin{equation}
\begin{split}
 \ch{\mathcal{V}_{[\alpha_{2,0}]}} &= \ch{W_{1,3}},\qquad \qquad\qquad\ \ \,
 \ch{\mathcal{V}_{[\alpha_{1,0}]}} \,=\, \ch{W_{2,3}},  \\
\ch{\mathcal{V}_{[\alpha_{1,1}]}} &= \ch{W_{1,1}} + \ch{W_{2,2}},\qquad      
\ch{\mathcal{V}_{[\alpha_{2,1}]}} \,=\, \ch{W_{2,1}} +\ch{W_{1,2}}, \\       
\ch{\mathcal{V}_{[\alpha_{1,-1}]}} &= \ch{W_{1,1}} + \ch{W_{2,2}},\qquad      
\ch{\mathcal{V}_{[\alpha_{2,-1}]}} = \ch{W_{2,1}} +\ch{W_{1,2}}. \\   
\end{split}
\end{equation}
This demonstrates that the $\mathcal{A}_3$-characters built from the $\TypMod{\lambda}^s$ decompose as a $V_{D_-}$-character times a $\cW(3)$-character.  Similarly, the non-standard irreducibles $\IrrMod{0}$ and $\IrrMod{-2/3}$ give rise, via \eqref{ch:p=3Irr}, to
\begin{equation}
\begin{split}
\sum_{s \in 3 \ZZ} \ch{\IrrMod{0}^{s+r}} &= 
\sum_{m,s \in \ZZ} \ch{\mathsf{F}_{2m-2-4r/3-4s}} \cdot \ch{M_{m,1}}, \\
\sum_{s \in 3 \ZZ} \ch{\IrrMod{-2/3}^{s+r}} &= 
\sum_{m,s \in \ZZ} \ch{\mathsf{F}_{2m-8/3-4r/3-4s}} \cdot \ch{M_{m,2}}. \\
\end{split}
\end{equation}
Simplifying now gives

\begin{prop}
Characters of $\mathcal A_3$-modules decompose into 
$V_{D_-}\otimes \AKMA{sl}{2}_{-4/3}$-characters as
\begin{equation}
\begin{split}
\sum_{s\in 3\ZZ} \ch{\TypMod{0}^{s+r}} &=\ch{\mathsf{V}_{[-4r/3]}}\cdot \ch{W_{1,3}}+
\ch{\mathsf{V}_{[2-4r/3]}}\cdot \ch{W_{2,3}}, \\
\sum_{s \in 3 \ZZ} \ch{\IrrMod{0}^{s+r}} &= 
\ch{\mathsf{V}_{[2-4r/3]}} \cdot \ch{W_{2,1}}+\ch{\mathsf{V}_{[-4r/3]}} \cdot \ch{W_{1,1}},\\
\sum_{s \in 3 \ZZ} \ch{\IrrMod{-2/3}^{s+r}} &= 
\ch{\mathsf{V}_{[4/3-4r/3]}} \cdot \ch{W_{2,2}}+\ch{\mathsf{V}_{[-2/3-4r/3]}} \cdot \ch{W_{1,2}}.
\end{split}
\end{equation}
Here $r\in\{0,1,2\}$.
\end{prop}

\noindent This result, together with Theorem \ref{thm:triplet}, suggests that
$\mathcal A_3\cong \ker_{V_D}(Q_-)$ because
\begin{equation}
\ch{\mathcal A_3}=\sum_{s \in 3 \ZZ} \ch{\IrrMod{0}^{s}} = \ch{\mathsf{V}_{[0]}} \cdot \ch{W_{1,1}}+\ch{\mathsf{V}_{[2]}} \cdot \ch{W_{2,1}}=\ch{\ker_{V_D}(Q_-)}.
\end{equation}

\end{document}